%% file: main.tex
\documentclass[11pt, a4paper]{amsart}
\usepackage[T1]{fontenc}

\usepackage[margin=3cm, top=4cm, bottom=3cm]{geometry}
  
\usepackage[dvipsnames]{xcolor}
\usepackage{xcolor-solarized}
\usepackage{amsthm}
\usepackage{amscd}
\usepackage{calc} 

\usepackage{amsmath,amssymb}
\usepackage{mathtools}
\usepackage{dsfont}

\usepackage{graphicx} 

\usepackage[colorlinks=true, allcolors=blue]{hyperref}
\usepackage{cleveref}

\usepackage{enumitem}
\usepackage{comment}
\usepackage{caption}
\usepackage{subcaption}

\usepackage[arrow, matrix, curve]{xy}
\usepackage{tikz}
\usepackage{tikz-cd}
\usetikzlibrary{
	cd,
	calc,
	positioning,
	fit,
	arrows,
	decorations.pathreplacing,
	decorations.markings,
	shapes.geometric,
	backgrounds,
	bending
}
\usepackage{tikzsymbols}
\usepackage{physics}

\usepackage{newtxtext,newtxmath}

\newcommand{\Z}{\mathbb{Z}}

\newcommand{\cS}{\mathcal{S}}

\newcommand{\cX}{\mathcal{X}}

\newcommand{\fF}{\mathfrak{F}}

\let\ker\relax
\DeclareMathOperator{\ker}{Ker}

\DeclareMathOperator{\sgn}{sgn}
\DeclareMathOperator{\Hom}{Hom}

\DeclareMathOperator{\crossing}{cr}

\DeclareMathOperator{\red}{red}
\DeclareMathOperator{\tor}{Tor}

\usepackage{amsthm}
\usepackage{thmtools}
\declaretheoremstyle[notefont=\bfseries,notebraces={}{},%
    headpunct={},postheadspace=1em]{mystyle}
\declaretheorem[style=mystyle,numbered=no,name=Theorem]{thm-hand}

\declaretheorem[style=plain, parent=section]{theorem}
\declaretheorem[sibling=theorem]{definition}
\declaretheorem[sibling=theorem]{lemma}

\declaretheorem[sibling=theorem]{corollary}
\declaretheorem[sibling=theorem]{proposition}

\declaretheorem[sibling=theorem, style=remark]{remark}
\declaretheorem[sibling=remark]{example}

\begingroup\expandafter\expandafter\expandafter\endgroup
\expandafter\ifx\csname pdfsuppresswarningpagegroup\endcsname\relax
\else
\fi

\definecolor{amaranth}{rgb}{0.9, 0.17, 0.31} 
\definecolor{carrotorange}{rgb}{0.80, 0.5, 0.01} 
\definecolor{citrine}{rgb}{0.89, 0.82, 0.04} 
\definecolor{dartmouthgreen}{rgb}{0.05, 0.5, 0.06} 
\definecolor{ballblue}{rgb}{0.13, 0.67, 0.8} 
\definecolor{ceruleanblue}{rgb}{0.16, 0.32, 0.75} 
\definecolor{amethyst}{rgb}{0.6, 0.4, 0.8} 
\definecolor{amber}{rgb}{1.0, 0.75, 0.0} 
\definecolor{burlywood}{rgb}{0.87, 0.72, 0.53} 

\usepackage{cite}

\begin{document}


\title{Reduced even and odd Khovanov homology of positive and $0$-adequate links}
\author[N\, Manikandan]{Naageswaran Manikandan}
\address{Max Planck Institute for Mathematics in Bonn}
\email{manikandan@mpim-bonn.mpg.de}


\date{} 


\begin{abstract}
In this article, we study reduced even and odd Khovanov homology, with a particular focus on positive and $0$-adequate links. For positive links, we show that the first reduced homology is supported in a single quantum grading, where it is free abelian, and that its rank is determined by the Seifert graph of any positive diagram. As a consequence, reduced even and odd Khovanov homology, as well as unreduced odd Khovanov homology, detect fiberedness among positive links, extending the previously known result for unreduced even Khovanov homology. We further show that the torsion in these three theories does not detect fiberedness in the same way as the torsion in unreduced even Khovanov homology.

We further show that $(p,q)$-cables of positive knots with $q\geq p$ exhibit the same behavior, even though such cables need not themselves be positive. We also study the stable homotopy types associated to reduced even and odd Khovanov homologies in the extremal and almost extremal quantum gradings of $0$-adequate links. At the almost extremal quantum grading, we show that the reduced stable homotopy types are wedges of sphere spectra. This contrasts with the unreduced setting, where the corresponding homotopy types depend on whether the $0$-state graph is bipartite.
\end{abstract}

\makeatletter
\@namedef{subjclassname@2020}{%
  \textup{2020} Mathematics Subject Classification}
\makeatother


\maketitle
\tableofcontents

\input{Introduction}
\input{Acknowledgement}
\input{Intro_OddKH}

\input{RedEvenKH_PositiveLinks}

\input{Extreme_AlmostextremeKH.tex}
\input{RedoddKH_PositiveLinks}

\input{Extreme_AlmostextremeOddKH.tex}

\input{General_coefficients.tex}
\input{Cables_OddKH}

\bibliographystyle{hamsalpha}
\bibliography{bibliography_knot3}

\end{document}

%% file: Introduction.tex
\section{Introduction}

Khovanov homology, introduced by M.~Khovanov in \cite{Khovanov_homology}, is a categorification of the Jones polynomial and has emerged as a powerful invariant in low-dimensional topology, see~\cite{Rasmussen_sInvaraint,KM_KhUnknotDetection,BaldwinSivekXie_HopflinkDetection,BaldwinSivek_Trefoildetection, ren2024khovanovhomologyexotic4manifolds}. In addition to the original (even) theory, an odd variant was constructed by P.~Ozsv{\'a}th, J.~Rasmussen, and Z.~Szab{\'o} in \cite{ORS_OddKhovanovHomology} which retains much of the formal structure of the even theory while detecting genuinely different phenomena. Both theories admit reduced versions, obtained by the choice of a basepoint, which are often more amenable to computation and exhibit sharper structural features.

In this paper, we study the reduced even and odd Khovanov homologies of positive links and $0$-adequate links. Our approach is motivated by prior work on the structure of Khovanov homology for these classes, in particular \cite{Stosic_PositiveBraids,semi_adequate_torsion, extremeKH, almost_extremeKH, KMMS_Khovanov_pos_fib}, where strong constraints on the extremal and near-extremal gradings are established.

A guiding objective of this work is to understand to what extent \emph{positivity} of a link can be detected purely from Khovanov homology. This motivation is inspired by analogous results in Floer-theoretic settings, notably the characterization of strongly quasipositive fibered knots via the Ozsv{\'a}th--Szab{\'o} $\tau$-invariant, as established by M.~Hedden in \cite{Hedden_Positivitynotion}. In this context, one seeks structural features of a homological invariant that are rigid enough to encode the positivity of the underlying link.

More concretely, it is natural to expect that the behavior of the initial homological gradings in Khovanov homology carries information about positivity. In this paper, we take steps in this direction by studying the reduced even and odd Khovanov homologies. The reduced framework proves particularly well-suited for the odd Khovanov homology, allowing for a direct comparison with the even case and revealing structural similarities and differences between the two theories.

Another major theme of this work is the study of \emph{stable homotopy types} associated with reduced Khovanov homology, both even and odd. Since their constructions in \cite{LS_StableHomotopyType,SSS_OddStableHomotopyType}, these homotopy types have attracted interest as they provide finer topological information than homology alone~\cite{LS_SteenrodSquares}. While a full characterization of positivity via stable homotopy types is beyond the scope of this paper, we take initial steps in this direction by analyzing their structure in extremal and almost-extremal quantum gradings. In doing so, we generalize results from \cite{almost_extremeKH}, which addressed the unreduced even Khovanov homology of positive links. Since all positive link diagrams are $0$-adequate, these findings naturally apply to positive diagrams as well, and offer a foundation for further study of homotopy-theoretic invariants in the context of positivity.

\subsection{Khovanov homology of positive links}

\begin{figure}[h!]
	\centering
	\includegraphics[width=.30\textwidth]{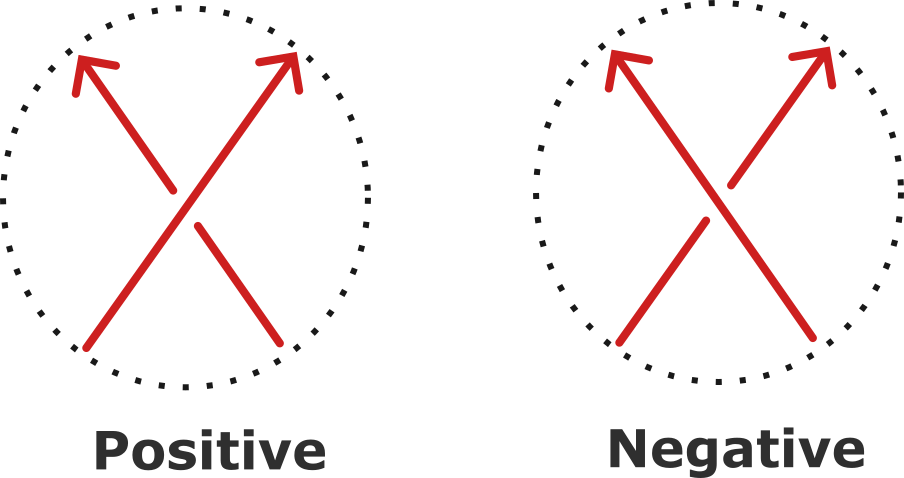}
	\caption{Crossing sign convention.}
	\label{fig:crossingsign}
\end{figure}

Given an oriented link diagram, each crossing can be classified as either positive or negative using the convention in Figure~\ref{fig:crossingsign}. A diagram is called \emph{positive} (resp. negative) if all of its crossings are positive (resp. negative).

\begin{definition}[Positive links]
    A link is called \emph{positive} if it admits a positive diagram.
\end{definition}

A central result of this paper is the following theorem.

\begin{theorem} \label{thm: maintheorem1}
Let $L$ be a non-split positive link, with a base point $p$. Then
\begin{equation*}
Kh_{e,\red}^{1,j}(L,p) \cong Kh_{o,\red}^{1,j}(L) =
\begin{cases}
\Z^{p_1(L)} & \text{if } j = 3 - \chi(L),\\
0 & \text{otherwise,}
\end{cases}
\end{equation*}
where $\chi(L)$ denotes the Euler characteristic of a genus-minimizing Seifert surface of $L$, and $p_1(L)$ is the first cyclomatic number of the reduced Seifert graph of any positive diagram of $L$ (see Section~\ref{sec: Khovanov homology} for a precise definition).
\end{theorem}


Several structural properties of reduced Khovanov homology for positive links are already known. In particular, if $L$ is a non-split positive link, then
\begin{align*}
    Kh_{*,\red}^{i,j}(L,p)&=0 \,\textrm{ if } \,i<0 \,\textrm{ or }\, j<1-\chi(L),\\
    Kh_{*,\red}^{0,j}(L,p)&=\begin{cases}
        \Z &\textrm{ if } \,j=1-\chi(L),\\
    0  &\textrm{ otherwise, } \\
    \end{cases}\\
    Kh_{*,\red}^{i,1-\chi(L)}(L,p)&=0 \, \textrm{ if }\,  i\neq0,
\end{align*}
where $*\in \{e,o\}$. The first two statements follow directly from the construction of Khovanov homology. The third, while well known to experts, does not appear to have a readily accessible reference in the literature; for completeness, we include a proof in Section~\ref{sec: Reduced even Khovanov homology of positive links}.

Theorem~\ref{thm: maintheorem1} is closely aligned in spirit with a result in \cite[Theorem~1.1]{KMMS_Khovanov_pos_fib}, which establishes that the unreduced even Khovanov homology of a non-split positive link~$L$ satisfies
\begin{equation}\label{eqn:UnreducedEvenKH_positivity}
Kh_{e}^{1,j}(L) \cong
\begin{cases}
\mathbb{Z}^{p_1(L)} & \text{if } j = 2 - \chi(L), \\
0 & \text{otherwise.}
\end{cases}
\end{equation}
In this paper, we extend this characterization to both the reduced even and odd Khovanov homologies. 
A natural question is to ask how the unreduced odd Khovanov homology looks in this context. This can be deduced from the identity
\[
Kh_o^{i,j}(L) \cong Kh_{o,\mathrm{red}}^{i,j-1}(L) \oplus Kh_{o,\mathrm{red}}^{i,j+1}(L),
\]
established in \cite{ORS_OddKhovanovHomology}, yielding the following result.

\begin{theorem}\label{thm: maintheorem2}
For a non-split positive link $L$, the unreduced odd Khovanov homology satisfies
\begin{equation*}
\begin{aligned}
Kh_{o}^{1,j}(L) &=
\begin{cases}
\Z^{p_1(L)} & \text{if } j = 2-\chi(L) \text{ or } 4-\chi(L),\\
0 & \text{otherwise.}
\end{cases}
\end{aligned}
\end{equation*}
where $\chi(L)$ and $p_1(L)$ are defined as before.
\end{theorem}

An immediate consequence of these results is the following characterization of fibered positive links.

\begin{corollary}\label{cor: maincorollary3}
Let $L$ be a non-split positive link with basepoint $p$. Then the following are equivalent: $L$ is fibered; $Kh_{e,\red}^{1,*}(L,p) = 0$; $Kh_{o,\red}^{1,*}(L) = 0$; $Kh_{e}^{1,*}(L) = 0$ and $Kh_{o}^{1,*}(L) = 0$.
\end{corollary}

\begin{proof}
The results for reduced and odd homology follow immediately from Theorems~\ref{thm: maintheorem1} and \ref{thm: maintheorem2}, together with A.~Stoimenov's observation \cite{Stoimenow_PolynomailValues_Unknottingnumbers} that a positive link diagram $D$ has $p_1(D) = 0$ if and only if the link it represents is fibered. The result for even Khovanov homology follows from \cite[Theorem~1.2]{KMMS_Khovanov_pos_fib}.
\end{proof}
This shows that both reduced even and odd Khovanov homologies detect fiberedness among positive links.
In contrast, the relation
\[
Kh_e^{i,j}(L) = Kh_{e,\red}^{i,j-1}(L,p) \oplus Kh_{e,\red}^{i,j+1}(L,p)
\]
does not hold in general. As a result, the structure of the unreduced even Khovanov homology in Equation~\eqref{eqn:UnreducedEvenKH_positivity} differs from that of the odd Khovanov homology in Theorem~\ref{thm: maintheorem2}.

In \cite[Theorem~1.1]{KMMS_Khovanov_pos_fib}, the torsion subgroup of unreduced even Khovanov homology of a non-split positive link~$L$ was shown to satisfy
\begin{equation}
\label{eq:torsion-even-khovanov}
\tor(Kh_{e}^{2,j}(L)) \cong
\begin{cases}
\Z_2^{p_1(L)} & \text{if } j = 4 - \chi(L), \\
0 & \text{otherwise,}
\end{cases}
\end{equation}
where $\tor(\cdot)$ denotes the torsion subgroup.
The same argument as in the proof of Corollary~\ref{cor: maincorollary3} shows that the torsion in unreduced even Khovanov homology detects fiberedness for positive links. This property, however, does not extend to unreduced odd Khovanov homology or to reduced even and odd Khovanov homology.

\begin{theorem}
Let $L$ be a non-split positive link with basepoint $p$. Then, for every $j$,
\[
\tor(Kh_{e,\red}^{2,j}(L,p)) = \tor(Kh_{o,\red}^{2,j}(L)) = \tor(Kh_{o}^{2,j}(L)) = 0,
\]
where $\tor(\cdot)$ denotes the torsion subgroup.
\end{theorem}

This is particularly interesting in light of \cite{ebert2025localmathfrakgl11actionoddkhovanov}, where it was shown that arbitrary torsion can occur in odd Khovanov homology, whereas the corresponding question remains open for unreduced even Khovanov homology. One might therefore expect odd Khovanov homology to exhibit richer torsion phenomena in the setting considered here; however, our results show that this is not the case for positive links.

In particular, this shows that results concerning torsion in unreduced even Khovanov homology, such as those in \cite{Mukherjee_Torsion,Shumakovitch_TorsionThinlinks,semi_adequate_torsion,MPSWY_Search_for_torsion}, need not extend to unreduced odd Khovanov homology.

Moreover, $(p,q)$-cables of positive knots with $q \geq p$ exhibit the same concentration behavior, even if they are not strictly positive.

\begin{theorem}\label{thm: maintheorem4}
    If $K$ is a positive knot, then for every $q\geq p\geq2$ the Khovanov homology of its $(p,q)$-cable $K_{p,q}$ fulfills 
    \begin{align*}
    Kh_{*,\red}^{i,j}(K_{p,q})&=0 \,\textrm{ if } \,i<0,\\
    Kh_{*,\red}^{0,j}(K_{p,q})&=\begin{cases}
        \Z &\textrm{ if } \,j=1-\chi(K_{p,q}),\\
    0  &\textrm{ otherwise, } \\
    \end{cases}\\
    Kh_{*,\red}^{1,j}(K_{p,q})&=\begin{cases}
        \Z^{p_1(K)} &\textrm{ if } \,j=3-\chi(K_{p,q}) ,\\
    0  &\textrm{ otherwise, } 
    \end{cases}\\
    \tor(Kh^{2,j}_{*,\red}(K_{p,q})) &=0 \quad \text{for every} \quad j,
\end{align*}
where $*\in\{e,o\}$, $\chi(L)$, $p_1(L)$ and $\tor(\cdot)$ are defined as before. Here, we suppress the base point from the notation of reduced Khovanov homology to avoid unnecessary clutter, since the result is independent of its choice.
\end{theorem}

\subsection{Khovanov homology of $0$-Adequate Links}

$0$-adequate links form an important class of links characterized by the property that the graph associated to the all-$0$ resolution of a diagram contains no loops (see Section~\ref{sec: Khovanov homology defn} for more details). For such links, we describe both reduced homology and the associated stable homotopy type at almost-extremal quantum gradings.

\begin{theorem}\label{thm: maintheorem5}
Let $D$ be a $0$-adequate diagram of a non-split link with $n_+$ positive and $n_-$ negative crossings and basepoint $p$. Then
$$
Kh_{e,\red}^{i, n_+ - 2n_- - |s_0 D| + 3}(D,p) \cong Kh_{o,\red}^{i, n_+ - 2n_- - |s_0 D| + 3}(D,p)=
\begin{cases}
\Z^{p_1(D)} & \text{if } i = -n_- +1,\\
0 & \text{otherwise.}
\end{cases}
$$
\end{theorem}

The following corollary follows immediately from the preceding theorem and the defining properties of the even and odd Khovanov stable homotopy types in \cite{LS_StableHomotopyType,SSS_OddStableHomotopyType}. See also Section~\ref{sec: Extreme reduced even Khovanov homology} for a brief overview of these stable homotopy types.

\begin{corollary}\label{thm: maincorollary6}
Let $D$ be a $0$-adequate diagram of a non-split link with $n_+$ positive and $n_-$ negative crossings and basepoint $p$. Then the reduced even and odd Khovanov stable homotopy types at quantum grading $j = n_+ - 2n_- - |s_0D| + 3$ are both wedge sums of spheres of the form $(S^1, 1-n_-)$.
\end{corollary}

These results show that, in the reduced setting, the almost extremal homology and stable homotopy types are independent of combinatorial complications such as the bipartiteness of the associated graph $G_0(D)$, which contrasts with previous findings in the unreduced case \cite{almost_extremeKH}. 

As a consequence of Theorem~\ref{thm: maintheorem5} and the fact that positive links are $0$-adequate, we obtain the following proposition.

\begin{proposition} \label{thm: mainproposition7}
Let $L$ be a non-split positive link, with a base point $p$. Then
\begin{equation*}
Kh_{e,\red}^{i,\,3-\chi(L)}(L,p) \cong Kh_{o,\red}^{i,\,3-\chi(L)}(L) =
\begin{cases}
\Z^{p_1(L)} & \text{if } i = 1,\\
0 & \text{otherwise.}
\end{cases}
\end{equation*}
where $\chi(L)$ and $p_1(L)$ are defined as before.
\end{proposition}

\subsection*{Outline}
The paper is organised as follows. In Section~\ref{sec: Khovanov homology}, we recall the definitions of even and odd Khovanov homology. Sections~\ref{sec: Reduced even Khovanov homology of positive links} and \ref{sec: Extreme and almost extreme reduced even Khovanov homology} develop the theory for even Khovanov homology and prove Theorems~\ref{thm: maintheorem1} and \ref{thm: maintheorem5}. Sections~\ref{sec: Reduced odd Khovanov homology of positive links} and \ref{sec: Extreme and almost extreme reduced odd Khovanov homology} establish the corresponding results for odd Khovanov homology. In Section~\ref{sec: Reduced Khovanov homology over other coefficent rings}, we extend the preceding arguments to general coefficient rings, which allows us to establish the results concerning torsion.
Finally, Section~\ref{sec: Reduced Khovanov homology of cables} studies reduced Khovanov homology of cables, with results for both even and odd Khovanov homology.

%% file: Acknowledgement.tex
\subsection*{Acknowledgement}
I would like to thank the Max-Planck Institute for Mathematics in Bonn for its hospitality and financial support. I would like to thank Léo Schelstraete for explaining subtle details in the construction of odd Khovanov homology and for pointing out errors in earlier versions of the proofs in this article. I would like to thank Tomás Mejía Gómez, Shun Ming Samuel Lee and Claudius Zibrowius for helpful suggestions that improved the exposition.

\subsection*{AI declaration}
AI was used to refine writing, particularly to correct grammar and improve sentence structure, and not for any other purpose.

%% file: Intro_OddKH.tex
\section{Even and odd Khovanov homology} \label{sec: Khovanov homology}

In this section, we review the definitions and basic properties of \emph{even} and \emph{odd Khovanov homology}, along with their reduced versions. Our exposition closely follows \cite{SSS_OddStableHomotopyType}.

\subsection{Algebraic Preliminaries}

Before defining the even and odd Khovanov chain complexes, we introduce two $\mathbb{Z}$-modules that form the algebraic input for both theories. These modules are defined for an arbitrary finite set $S$; in later sections, $S$ will typically be a finite collection of circles in $\mathbb{R}^2$.

Let $S$ be a finite set, and let $T(S)$ denote the tensor algebra generated by $S$ over $\mathbb{Z}$.

\begin{definition}[Even algebra]\label{def: even module}
	Let $I_e \subset T(S)$ be the two-sided ideal generated by
	\[
	x \otimes x \quad \text{and} \quad x \otimes y - y \otimes x,
	\qquad x,y \in S.
	\]
	The \emph{even algebra} associated to $S$ is defined by
	\[
	\Lambda_e(S) := T(S)/I_e.
	\]
\end{definition}

\begin{definition}[Odd algebra]\label{def: odd module}
	Let $I_o \subset T(S)$ be the two-sided ideal generated by
	\[
	x \otimes x \quad \text{and} \quad x \otimes y + y \otimes x,
	\qquad x,y \in S.
	\]
	The \emph{odd algebra} associated to $S$ is defined by
	\[
	\Lambda_o(S) := T(S)/I_o.
	\]
\end{definition}

\begin{remark}
	\begin{enumerate}
	\item The algebra $\Lambda_e(S)$ is the symmetric algebra on $S$ modulo the ideal generated by the squares of the generators. The algebra $\Lambda_o(S)$ is the exterior algebra generated by $S$. 
	
	\item If $S$ is a singleton, then $\Lambda_e(S) \cong \mathbb{Z}[x]/(x^2)$. More generally, for a finite set $S$, we have
	\[
	\Lambda_e(S) \cong \left(\mathbb{Z}[x]/(x^2)\right)^{\otimes |S|}.
	\]
	
	\end{enumerate}
	
\end{remark}

\subsection{Even and Odd Khovanov homology}\label{sec: Khovanov homology defn}
In this section, we recall the definitions of even and odd Khovanov homology.

\begin{definition}[States of a diagram]
	Let $D$ be an oriented link diagram with $n$ ordered crossings $\crossing(D) = \{c_1, \dots, c_n\}$. A \emph{state} of a diagram \(D\) is a function  $s \colon \crossing(D) \to \{0,1\},$ assigning to each crossing of \(D\) a value \(0\) or \(1\), called a \emph{marker}.
\end{definition}

Let $\mathcal{S}(D)$ denote the set of all $2^n$ possible states of the diagram $D$. Given $s \in \cS(D)$, the \emph{resolution} $sD$ corresponds to the diagram obtained after smoothing every crossing $c_i \in \crossing(D)$ according to its marker $s(c_i)$ following Figure~\ref{fig: smoothing}. The resolution $sD$ does not contain any crossings and consists of a collection of $|sD|$ circles embedded in the plane.

\begin{definition}[Seifert state and Seifert resolution]
	The \emph{Seifert state} is the state in which each positive crossing is assigned the marker $0$ and each negative crossing is assigned the marker $1$.
	The \emph{Seifert resolution} is the resolution of a diagram corresponding to the Seifert state.
\end{definition}

\begin{figure}[h!]
	\centering
	\includegraphics[width=.50\textwidth]{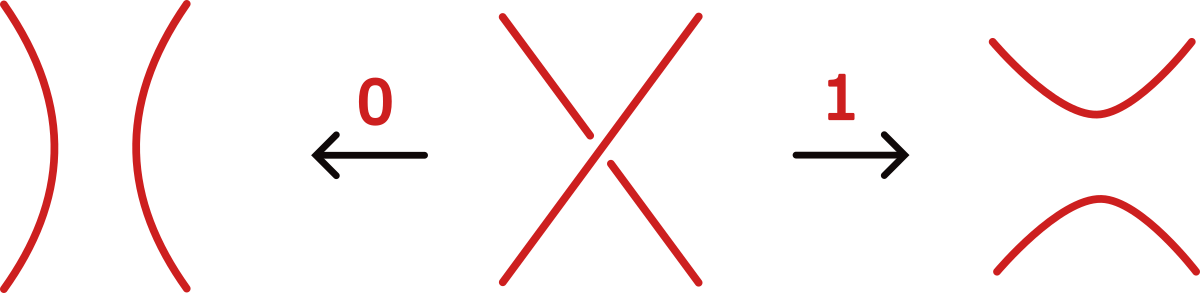}
	\caption{Smoothing.}
	\label{fig: smoothing}
\end{figure}

\begin{definition}[State graph]
	Given a state $s\in \cS(D)$, the \textit{state graph} $G_s(D)$ is a planar graph obtained by assigning a vertex to each circle in $sD$ and adding an edge between two vertices for each crossing connecting the corresponding circles in the original diagram.
\end{definition} 

The graph $G_{s}(D)$ may contain loops and/or multi-edges (i.e. edges connecting the same pair of vertices, also called parallel edges). 

\begin{definition}[Reduced graph]
Given a graph $G$, its associated reduced graph $G^{\text{red}}$ is the graph obtained from $G$ by identifying parallel edges.
\end{definition}

Let $s_0$ and $s_1$ be the states assigning a $0$-marker and a $1$-marker to every crossing, respectively, with $G_0(D)$ and $G_1(D)$ their corresponding state graphs. The state graph of the Seifert state is called the \emph{Seifert graph}. For a positive link, the Seifert graph coincides with the state graph $G_0(D)$.

\begin{definition}[Adequate diagrams]
	A link diagram $D$ is called \textit{$0$-adequate} if the state graph $G_0(D)$ contain no loops,  \textit{$1$-adequate} if the state graph $G_1(D)$ contain no loops and \textit{adequate} if both $G_0(D)$ and $G_1(D)$ contains no loops.
\end{definition}

\begin{example}
	Every positive diagram is \(0\)-adequate.
\end{example}

\begin{definition}[Cyclomatic number] \label{def: cyclomatic number}
	The \textit{cyclomatic number }$p_1(G)$ of a graph $G$ is the minimal number of edges that one must delete to transform $G$ into a collection of disjoint tree graphs. It can be computed as $p_1(G) = e-v+k$, where $v, e$ and $k$ denote the number of vertices, edges and connected components of $G$, respectively. 
	
	Given a link diagram $D$, we define its cyclomatic number $p_1(D)$ as the cyclomatic number of the reduced $G_0(D)^{\red}$ graph associated with $D$. 
\end{definition}

When defining odd Khovanov homology, in addition to fixing an ordering of the crossings, we must equip the oriented link diagram $D$ with an additional decoration. Specifically, we choose an \emph{orientation of each crossing}, i.e., an assignment of an arrow to every crossing of $D$, as illustrated in Figure~\ref{fig: orientation of crossings}.

\begin{figure}[h!]
	\centering
	\includegraphics[width=.10\textwidth]{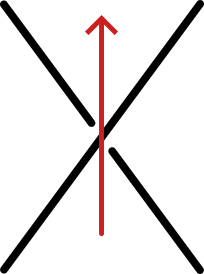}
	\caption{Orientation of crossings.}
	\label{fig: orientation of crossings}
\end{figure}

Given a state $s \in \cS(D)$ of a decorated diagram $D$ with $n$ ordered crossings, the \emph{resolution} $sD$ is obtained by smoothing each crossing $c_i \in \crossing(D)$ according to the marker $s(c_i)$, as illustrated in Figure~\ref{fig: smoothing_oriented}. The resolution $sD$ of a decorated diagram contains no crossings. It consists of a collection of $|sD|$ disjoint circles embedded in the plane, together with $n$ oriented arcs placed between the circles at the points corresponding to the original crossings.

\begin{figure}[h!]
	\centering
	\includegraphics[width=.50\textwidth]{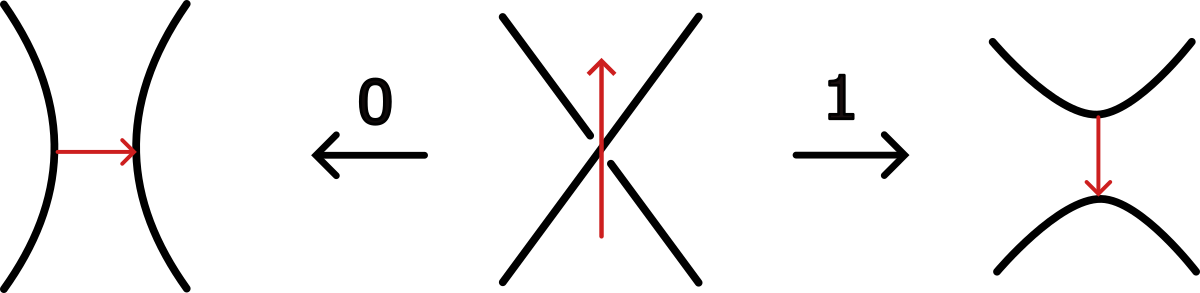}
	\caption{Smoothing of an oriented crossing.}
	\label{fig: smoothing_oriented}
\end{figure}

\begin{remark}
We will not distinguish, in our notation, between states or their resolutions for decorated and undecorated diagrams. When discussing even Khovanov homology, all diagrams are understood to be undecorated, whereas in the odd Khovanov homology setting, diagrams are assumed to be decorated.
\end{remark}

The set of states $\mathcal{S}(D)$ is naturally identified with $\mathbf{2}^n \coloneq \{0,1\}^n$, the vertices of the $n$-dimensional unit cube. This identification is defined by associating to $s = (s_1, \ldots, s_n) \in \{0,1\}^n$ the state $s(c_j) = s_j$ for each crossing $c_j$. 
We order the set $\{0,1\}$ by declaring $0<1$, which induces the product partial order on
$\mathbf{2}^n$. The cube is graded by
\[
|s|=\sum_i s_i.
\]
We write $s\leq_k s'$ if $s\leq s'$ and $|s'|-|s|=k$. Two states $s$ and $s'$ are connected by an edge in the $n$-dimensional unit cube if and only if either $s\leq_1 s'$ or $s'\leq_1 s$. 

\begin{remark}
A categorical structure induced on $\mathbf{2}^n$ by the partial order, where $\Hom_{\underline{2}^n}(s,s')$ has a single element if $s \leq s'$ and is empty otherwise.  Write $\phi_{s,s'}$ for the unique morphism $s \to s'$ if it exists. We will denote this by $\underline{2}^n$ when we regard cube as a catgeory, rather than as a poset or simply as a set.
\end{remark}

\subsubsection{Even Khovanov functor}

With the cube category $\underline{2}^n$ in place, we define a functor
\[
\fF_e \colon \underline{2}^n \longrightarrow \mathsf{Ab}
\]
from the category $\underline{2}^n$ to the category of abelian groups $\mathsf{Ab}$, using the resolutions associated to the vertices of the $n$-dimensional cube.

\noindent\textbf{On objects.}
For a state $s\in \underline{2}^n$, set
\[
\fF_e(s) \coloneq \Lambda_e(sD).
\]
\noindent\textbf{On morphisms.}
It suffices to define $\fF_e$ on morphisms $\phi_{s,s'}$ with $s\leq_1 s'$, since morphisms between arbitrary states are obtained by composition along paths in the cube. For this assignment to be well defined, the resulting maps must commute around all $2$-dimensional faces of the cube.

Let $\phi_{s,s'}$ denote the unique morphism with $s\leq_1 s'$. The resolutions $sD$ and $s'D$ differ locally by either a split of a circle or a merge of two circles.

\begin{enumerate}
	\item Suppose $\phi_{s,s'}$ corresponds to a split, so that a circle $a\in sD$ splits into two circles $a_1,a_2\in s'D$, with all other circles naturally identified. Define
	\[
	\fF_e(\phi_{s,s'})(y) \coloneq (a_1+a_2)\otimes y,
	\]
	where $\Lambda_e(sD)$ is viewed as a subalgebra of $\Lambda_e(s'D)$ by identifying $a$ with either $a_1$ or $a_2$.
	
	\item Suppose instead that $\phi_{s,s'}$ corresponds to a merge, so that two circles $a_1,a_2\in sD$ merge into a single circle $a\in s'D$, again with all other circles naturally identified. Define
	\[
	\fF_e(\phi_{s,s'}) \colon \Lambda_e(sD) \longrightarrow \Lambda_e(s'D)
	\]
	to be the homomorphism sending both $a_1$ and $a_2$ to $a$ and acting as the identity on all other generators.
\end{enumerate}

The assignments above strictly commute around every $2$-dimensional face of the cube. Consequently, $\fF_e$ is a well-defined functor on the cube category $\underline{2}^n$.

\subsubsection{Odd Khovanov functor}

We now define the \emph{odd Khovanov functor}, which differs from the even functor $\fF_e$ only by sign choices on the morphisms.

As in the even case, we attempt to define a functor
\[
\fF'_o \colon \underline{2}^n \longrightarrow \mathsf{Ab}
\]
using the resolutions associated to the vertices of the $n$-dimensional cube. Note that a resolution $sD$ is not merely a collection of $|sD|$ disjoint circles embedded in the plane. Rather, it also carries oriented arcs between the circles at the locations corresponding to the original crossings. This additional structure will play a crucial role in the definition of the functor.

\noindent\textbf{On objects.}
For a state $s\in \underline{2}^n$, set
\[
\fF'_o(s) \coloneq \Lambda_o(sD).
\]

\noindent\textbf{On morphisms.}
As before, it suffices to define $\fF_o$ on morphisms $\phi_{s,s'}$ with $s\leq_1 s'$, since all other morphisms are obtained by composition. The resolutions $sD$ and $s'D$ differ locally by either a split of a circle or a merge of two circles.

\begin{enumerate}
	\item Suppose $\phi_{s,s'}$ corresponds to a split, so that a circle $a \in sD$ splits into two circles $a_1, a_2 \in s'D$, with all other circles naturally identified. Let the oriented arc at the crossing where $s$ and  $s'$ differ (resolved with the $1$ marker) go from $a_1$ to $a_2$. Then define
	\[
	\fF'_o(\phi_{s,s'})(y) \coloneq (a_1 - a_2) \otimes y,
	\]
	where $\Lambda_o(sD)$ is embedded in $\Lambda_o(s'D)$ by sending $a$ to either $a_1$ or $a_2$.
	
	\item Suppose instead that $\phi_{s,s'}$ corresponds to a merge, so that two circles $a_1,a_2\in sD$ merge into a single circle $a\in s'D$, again with all other circles naturally identified. Define
	\[
	\fF'_o(\phi_{s,s'}) \colon \Lambda_o(sD) \longrightarrow \Lambda_o(s'D)
	\]
	to be the homomorphism sending both $a_1$ and $a_2$ to $a$ and acting as the identity on all other generators.
\end{enumerate}

Unlike in the even case, the edge maps of $\fF'_o$ do not strictly commute around $2$-dimensional faces of the cube; instead, they commute only up to sign. Thus, this naive assignment does not define a functor. However, one can consistently choose signs for each edge of the $n$-dimensional cube so that, after multiplying by these signs, the resulting maps define a functor.

The $2$-dimensional configurations can be divided into four types (with unoriented arcs being orientable in either direction).

\begin{figure}[ht]
\centering
\begin{minipage}[t]{0.68\textwidth}
    \centering
	\vspace{0pt}
    \begin{subfigure}{0.45\linewidth}
        \centering
        \includegraphics[width=0.35\linewidth]{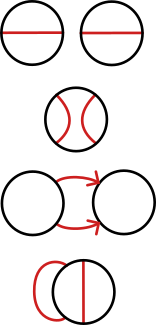}
        \caption{Type A}
        \label{fig:1}
    \end{subfigure}
    \hfill
    \begin{subfigure}{0.45\linewidth}
        \centering
        \includegraphics[width=0.67\linewidth]{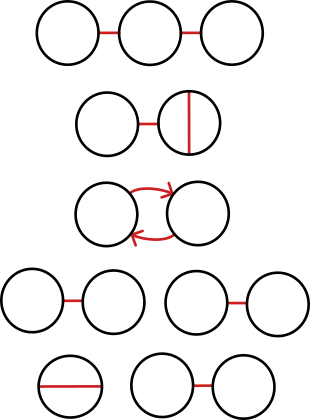}
        \caption{Type C}
        \label{fig:2}
    \end{subfigure}

\end{minipage}
\hfill
\begin{minipage}[t]{0.28\textwidth}
	\centering
	\vspace{1cm}
   	\begin{subfigure}{0.45\linewidth}
        \centering
        \includegraphics[width=0.5\linewidth]{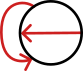}
        \caption{Type X}
        \label{fig:3}
    \end{subfigure}

    \vspace{0.5em}

    \begin{subfigure}{0.45\linewidth}
        \centering
        \includegraphics[width=0.5\linewidth]{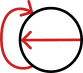}
        \caption{Type Y}
        \label{fig:4}
    \end{subfigure}
\end{minipage}

\caption{Types of 2-dimensional configurations.}
\label{fig:SquareTypes}
\end{figure}

The behavior of the edge maps of $\fF'_o$ around $2$-dimensional faces of the cube depends on the face type. Specifically:
\begin{itemize}
	\item For type-A faces, $\fF'_o$ commutes only after multiplication by $-1$.
	\item For type-C faces, $\fF'_o$ commutes strictly, without any sign adjustment.
	\item For type-X and type-Y faces, $\fF'_o$ commutes both strictly and up to a
	sign.
\end{itemize}

To systematically correct the signs, define cochains
\[
\psi_X, \psi_Y \in C^2_{\mathrm{cell}}([0,1]^n; \mathbb{Z}_2)
\]
by setting
\[
\psi_X(f) =
\begin{cases}
	1, & \text{if $f$ is a type-A or type-X face},\\
	0, & \text{if $f$ is a type-C or type-Y face},
\end{cases}
\qquad
\psi_Y(f) =
\begin{cases}
	1, & \text{if $f$ is a type-A or type-Y face},\\
	0, & \text{if $f$ is a type-C or type-X face}.
\end{cases}
\]

\begin{definition}[Edge assignment]\label{def:edge-assignment}
	A \emph{type-X} (respectively, \emph{type-Y}) \emph{edge assignment} for a diagram $D$ with oriented crossings is a cochain
	\[
	\epsilon \in C^1_{\mathrm{cell}}([0,1]^n; \mathbb{Z}_2)
	\]
	satisfying
	\[
	\delta \epsilon = \psi_X \quad (\text{respectively, } \delta \epsilon = \psi_Y).
	\]
\end{definition}

\begin{theorem}[\cite{ORS_OddKhovanovHomology}, Lemma~1.2]\label{thm:edge-assignment}
	Every diagram $D$ admits both type-X and type-Y edge assignments.
\end{theorem}

Fix a choice of edge assignment $\epsilon$ (either type-X or type-Y). For each edge $\phi_{s,s'}$ in the cube, define the signed edge map
\[
\fF_o(\phi_{s,s'}) \coloneqq (-1)^{\epsilon(\phi_{s,s'})} \fF'_o(\phi_{s,s'}).
\]
With these sign corrections, the assignment $\fF_o$ extends to a well-defined functor
\[
\fF_o \colon \underline{2}^n \longrightarrow \mathsf{Ab}.
\]

\subsubsection{Khovanov Chain Complex}

The construction of the Khovanov chain complex requires a coherent choice of signs on the edges of the $n$-dimensional cube $\underline{2}^n$.  The \emph{standard sign assignment} is a function
\[
\sgn \colon \{\text{edges of } \underline{2}^n\} \longrightarrow \mathbb{Z}_2
\]
defined as follows. Given an edge connecting states $s \leq_1 s'$, let $k \in \{1,\dots,n\}$ be the unique index for which $s_k < s'_k$. Then
\[
\sgn_{s,s'} \coloneq \sum_{i=1}^{k-1} s_i \;\bmod 2.
\]

Using this choice of sign assignment, we define, for each $\bullet \in \{e,o\}$, the chain complex associated to a link diagram $D$ by
\[
CKh_\bullet(D) \coloneq \bigoplus_{s \in \underline{2}^n} \mathfrak{F}_\bullet(s).
\]
The differential is given by
\[
\partial_\bullet(x)
\coloneq
\sum_{\substack{s' \\ s \leq_1 s'}}
(-1)^{\sgn_{s,s'}}
\,\mathfrak{F}_\bullet\big(\phi_{s,s'}(x)\big).
\]
where $x \in \Lambda_{\bullet}(sD)$.
The elements of $CKh_\bullet(D)$ are naturally bigraded by their homological grading $h$ and the quantum grading $q$. For an element
\[
x = a_1 \otimes \cdots \otimes a_k \in \Lambda_\bullet(sD), 
\qquad a_1, \dots, a_k \in sD,
\]
these gradings are defined by
\[
\begin{aligned}
	h(x) &\coloneq |s| - n_-, \\
	q(x) &\coloneq |sD| - 2k + |s| + n_+ - 2n_-,
\end{aligned}
\]
where $n_+$ and $n_-$ denote the numbers of positive and negative crossings of the diagram $D$, respectively.
The boundary maps $\partial_\bullet$ preserve the quantum grading and increase the homological grading by $1$.  
Consequently, the resulting homology, denoted $Kh_{\bullet}(D)$, inherits a natural bigrading. Its graded components are denoted $Kh^{i,j}_\bullet(D)$, where $i$ and $j$ denote the homological and quantum gradings, respectively.

The definition of $Kh^{i,j}_{\bullet}(D)$ depends on certain auxiliary choices, such as the ordering of the crossings, the orientations at crossings, and the choice of edge assignment.

\begin{theorem}[\cite{Khovanov_homology, ORS_OddKhovanovHomology}]
	The bigraded homology $Kh_{\bullet}^{i,j}(D)$ is independent of the choices involved in its construction and, up to isomorphism, depends only on the link $L$ represented by $D$. 
\end{theorem}
Thus, we obtain a well-defined link invariant, denoted by $Kh_{\bullet}^{i,j}(L)$.

\subsection{Reduced Khovanov homology} \label{sec: reduced Khovanov homology}

Let $p$ be a basepoint on the diagram $D$.  
For $\bullet \in \{e, o\}$, the \emph{reduced complex} $CKh_{\bullet,\mathrm{red}}(D, p)$ is defined as the subcomplex of $CKh_\bullet(D)$ consisting of elements of the form
\[
a \otimes y,
\]
where $a$ is a circle containing $p$ in a given resolution of $D$, and $y$ is any other element.  

The complex $CKh_{\bullet,\mathrm{red}}(D, p)$ inherits the homological grading from $CKh_\bullet(D)$, but its quantum grading is shifted. For an element $x = a \otimes y$, the quantum grading is
\[
q_{\red}(x) \coloneq q(x) + 1,
\]
where $q(x)$ is the formula for the quantum grading in the unreduced complex. 

\begin{definition}[Reduced Khovanov homology] \label{def: reduced_KH}
	For $\bullet \in \{e, o\}$, the \emph{reduced Khovanov homology} is defined as
	\[
	Kh^{i,j}_{\bullet, \mathrm{red}}(D, p) \coloneq H_i\left(CKh^{*,j}_{\bullet, \mathrm{red}}(D, p), \partial_\bullet \right).
	\]
\end{definition}

In the odd case, the reduced homology depends only on the link $L$ represented by $D$, and we denote it by $Kh^{i,j}_{o, \mathrm{red}}(L)$. 
In the even case, the reduced homology depends only on the link $L$ represented by $D$ and the component of $L$ containing the base point $p$. Accordingly, we denote it by $Kh^{i,j}_{e, \mathrm{red}}(L, p)$.

The odd Khovanov complex splits as a direct sum, see \cite[Proposition~1.7]{ORS_OddKhovanovHomology}:
\begin{equation} \label{eqn:oddKH_split}
	CKh^{i,j}_{o}(D) = CKh^{i,j-1}_{o, \mathrm{red}}(D) \oplus CKh^{i,j+1}_{o, \mathrm{red}}(D).
\end{equation}

In contrast, the even theory does not admit such a splitting. Instead, the chain complexes fit into a short exact sequence:
\[
0 \longrightarrow CKh^{i,j+1}_{e, \mathrm{red}}(D, p) 
\longrightarrow CKh^{i,j}_{e}(D) 
\longrightarrow CKh^{i,j-1}_{e, \mathrm{red}}(D, p) 
\longrightarrow 0.
\]

%% file: RedEvenKH_PositiveLinks.tex
\section{Reduced even Khovanov homology of positive links}\label{sec: Reduced even Khovanov homology of positive links}
This section is dedicated to the proof of Theorem~\ref{thm: maintheorem1} for even Khovanov homology. We begin by establishing the necessary setup.

Let $L$ be a non-split positive link with a basepoint $p$. Choose a positive diagram $D$ representing $L$ and, without loss of generality, assume that $p$ lies away from all crossings and their resolutions.
Let $D'$ be the diagram with $n'$ crossings obtained from a diagram $D$ by performing Seifert smoothing at all but one crossing between each pair of Seifert circles. Let $L'$ denote the link represented by $D'$. Consequently, $D'$ inherits a natural orientation from $D$, and with this orientation, $D'$ is a positive diagram. We keep the base point $p$ fixed in its original position.  

The diagram $D'$ has the following properties, which will be essential later.
\begin{enumerate}
	\item We have $|s_0 D| = |s_0 D'|$. This follows immediately from the fact that the Seifert circles obtained from $D$ and $D'$ are exactly the same.
	\item By \cite[Lemma~4.2]{KMMS_Khovanov_pos_fib}, the diagram $D'$ is alternating, and its signature satisfies
	\begin{equation}\label{eqn:signature_cyclomaticnumber_equality}
	\sigma(L') = p_1(D) = p_1(D').
	\end{equation}
\end{enumerate}

\begin{proposition}\label{prop: even technical}
The reduced even Khovanov homology of the diagram $(D', p)$ constructed above satisfies
\[
Kh^{1}_{e, \mathrm{red}}(D, p) = Kh^{1}_{e, \mathrm{red}}(D', p).
\]

\end{proposition}

\begin{proof}
	
	We suppress quantum grading shifts throughout, as neither the statement nor the argument depends on them. Since the diagram \( D \) contains no negative crossings, the homological grading \( i \) equals \( |s| \).
	
	As recalled in Section~\ref{sec: reduced Khovanov homology}, there is a natural inclusion of chain complexes
	\[
	CKh^{i}_{e,\mathrm{red}}(D,p) \hookrightarrow CKh^{i}_e(D).
	\]
	We therefore describe cycles in \( CKh^{i}_e(D) \); the same description applies to the reduced complex.
	Let \( G = G_0(D) \) be the state graph associated to the state \( s_0 \), which assigns the \(0\)-marker to every crossing. Let \( E = \{ e_1, \ldots, e_n \} \) be the edge set of \( G \), where each edge \( e_j \) corresponds to a crossing \( c_j \). Let \( l = |s_0 D| \) denote the number of circles in the resolution \( s_0 D \).
	
	The group \( CKh^1_e(D) \) is generated by states with exactly one \(1\)-marker, and decomposes as
	\[
	CKh^1_e(D) = \bigoplus_{i=1}^n A_i,
	\]
	where \( A_i =\Lambda_e(s^iD)\) and the state \( s^i \) assigning a \(1\)-marker at \( c_i \) and \(0\)-markers elsewhere. The proof of \cite[Theorem~4.1]{KMMS_Khovanov_pos_fib} establishes the following fact. If
	\[
	z = (z_i)_{i=1}^n \in CKh^1_e(D)
	\]
	satisfies \( \partial z = 0 \), then \( z_i = z_j \) whenever the edges \( e_i \) and \( e_j \) are parallel in \( G \).
	
	Therefore, for the purpose of computing \( Kh^1_{e,\red}(D,p) \), we may replace \( D \) by a diagram \( D' \) obtained by retaining exactly one crossing between each pair of Seifert circles and resolving all remaining crossings using the \(0\)-marker. In particular,
	\[
	Kh^1_{e,\red}(D,p) \cong Kh^1_{e,\red}(D',p).
	\]
	
\end{proof}

\begin{lemma} \label{lem: extreme even KH}
If $D$ is $0$-adequate link diagram with $n_+$ positive and $n_-$ negative crossings and a basepoint $p$, then its reduced even Khovanov homology satisfies the following.
\[
Kh_{e,\red}^{i,n_+ - 2n_- - |s_0 D| + 1}(D,p) =
\begin{cases}
	\mathbb{Z} & \text{if } i = -n_- \\
	0 & \text{otherwise.}
\end{cases}
\]
\end{lemma}

\begin{proof}
	Consider the generator
	\[
	x = a_1 \otimes a_2 \otimes \cdots \otimes a_{|s_0 D|}
	\]
	in $\Lambda_e(s_0 D)$, where each $a_i$ corresponds to a circle of $s_0 D$. Observe that this generator lies in the minimal possible quantum grading.
	
	Since $D$ is $0$-adequate, performing a $1$-smoothing at any crossing reduces the number of circles in the smoothing of that state compared to $s_0 D$. Consequently, $\partial x=0$. Furthermore, as there are no chain groups in lower homological degrees whose differential maps to $x$, the element $x$ represents a nontrivial homology class.
	The same argument shows that, for every state $s_0\leq s$, the quantum grading of any generator in $\Lambda_e(sD)$ is strictly greater than the minimal quantum grading by $2$. Hence, no additional generators in the required quantum grading arise from states obtained by $1$-smoothings.
	
	Hence, the chain complex in this minimal quantum grading is generated uniquely by $x$, and so is its homology.
\end{proof}

\begin{lemma} \label{lem: almost extreme even KH}
If $D$ is a positive diagram of $L$ with $n$ crossings, then its reduced even Khovanov homology satisfies the following.
\[
Kh_{e,\red}^{0,n - |s_0 D| + 3}(D,p) =0.
\]
\end{lemma}
\begin{proof}
    By \cite[Remark~6.4]{almost_extremeKH}, the unreduced even Khovanov complex admits a unique generating cycle in bidegree $(i,j) = (0, n - |s_0 D| + 3)$, explicitly given by
    \[
    \sum_{v \in V_{G_0(D)}} \sigma(v)\, a_1 \otimes \cdots \widehat{a_v} \cdots \otimes a_{|s_0D|},
    \]
    where $\sigma \colon V_{G_0(D)} \to \{\pm 1\}$ assigns opposite signs to adjacent vertices, and $\widehat{a_v}$ indicates omission of the corresponding tensor factor. For details regarding the derivation of this generator, see the proof of Theorem~\ref{thm: maintheorem5} in Section~\ref{sec: Almost extreme reduced odd Khovanov homology}.

    This element does not lie in the subcomplex
    \[
    CKh_{e,\red}^{0,\, n - |s_0 D| + 3}(D,p) \hookrightarrow CKh_{e}^{0,\, n- |s_0 D| + 2}(D),
    \]
    and therefore does not contribute to the reduced even Khovanov homology.
\end{proof}

\begin{lemma}\label{lem: reduced even KH last_group}
 The reduced even Khovanov homology of the diagram $(D',p)$ above satisfies 
	\begin{align*}
		Kh^{1, j}_{e,\red}(D',p)=\begin{cases}
			\Z^{p_1(D')} &\textrm{ if } j=n'-|s_0D|+3,\\
			0 &\textrm{ otherwise,}
		\end{cases}
	\end{align*}
\end{lemma}

\begin{proof}
	Recall that, by construction, $D'$ a positive, alternating diagram representing a non-split link. Without loss of generality, we may also assume that $D'$ is reduced. By \cite{Lee_KhovanovHomology}, the reduced Khovanov homology of a non-split alternating link is thin, i.e.,
	\[
	Kh^{i,j}_{e,\red}(L') \neq 0 \quad \text{only if} \quad j - 2i = \sigma(L').
	\]
	In particular, for $i = 1$, the reduced Khovanov homology can be nonzero only in the quantum grading
	\[
	j = \sigma(L') + 2.
	\]
	By Equation~\eqref{eqn:signature_cyclomaticnumber_equality} and Definition~\ref{def: cyclomatic number}, we obtain	\[
	\sigma(L') = p_1(D') = n' - |s_0 D'| + 1,
	\]
	so that
	\[
	Kh^{1,j}_{e,\red}(D',p) \neq 0 \quad \text{only if} \quad j = n' - |s_0 D'| + 3.
	\]
	We now prove that 
	\[
	Kh^{1, n'-|s_0 D| + 3}_{e,\red}(D',p) \cong \mathbb{Z}^{\,p_1(D')}.
	\] 
	We proceed by induction on $p_1(D')$.
	\begin{itemize}
		\item If $p_1(D') = 0$, then $G_0(D')$ is a tree, so $D'$ represents the unknot. In this case, the reduced Khovanov homology clearly satisfies the statement. 
		\item Now suppose that $p_1(D') > 0$. Let $v$ be a crossing of $D'$, and consider its $0$-smoothing $D_0$ and $1$-smoothing $D_1$. Then, there is a long exact sequence in reduced Khovanov homology~\cite{MPSWY_Search_for_torsion}.
		\[
		\longrightarrow 
		Kh_{e,\mathrm{red}}^{\frac{w_1 - w - 1}{2} + 1,\, \frac{3(w_1 - w) - 1}{2} + p_1(D') + 2}(D_1,p)
		\longrightarrow Kh_{e,\mathrm{red}}^{1, p_1(D') + 2}(D',p)
		\longrightarrow Kh_{e,\mathrm{red}}^{1, p_1(D_0) + 2}(D_0,p) \longrightarrow
		\]
		where $w$ and $w_1$ denote the writhes of $D'$ and $D_1$, respectively.
		\item Since $D_0$ is positive, alternating, and represents a non-split link, the induction hypothesis applies. Hence, 
		\[
		Kh_{e,\mathrm{red}}^{1, p_1(D_0) + 2}(D_0,p) \cong \mathbb{Z}^{\,p_1(D_0)} = \mathbb{Z}^{\,p_1(D') - 1}.
		\] 
		Note that $D_1$ is adequate. Hence, by Lemma~\ref{lem: extreme even KH}, we have
		\[
		Kh_{e,\mathrm{red}}^{\frac{w_1 - w - 1}{2} + 1,\, \frac{3(w_1 - w) - 1}{2} + p_1(D') + 2}(D_1,p) \cong \mathbb{Z}, \quad \text{and} \quad Kh_{e,\mathrm{red}}^{\frac{w_1 - w - 1}{2} + 2,\, \frac{3(w_1 - w) - 1}{2} + p_1(D') + 2}(D_1,p)=0.
		\]
		By Lemma~\ref{lem: almost extreme even KH}, we have
		\[
		Kh_{e,\mathrm{red}}^{0, p_1(D_0) + 2}(D_0,p) = 0.
		\]
	\end{itemize}

	Combining these facts with the long exact sequence, we obtain the short exact sequence
	\[
	0 \longrightarrow \mathbb{Z} \longrightarrow Kh_{e,\mathrm{red}}^{1, p_1(D') + 2}(D',p) \longrightarrow \mathbb{Z}^{\,p_1(D') - 1} \longrightarrow 0.
	\]
	
	This completes the induction step and proves the desired result for all $p_1(D')$.
\end{proof}

\begin{proof}[Proof of Theorem~\ref{thm: maintheorem1} (even case)]
From Proposition~\ref{prop: even technical} and Lemma~\ref{lem: reduced even KH last_group}, we obtain that the reduced even Khovanov homology of $D$ in homological grading $1$ is $\mathbb{Z}^{\,p_1(D)}$. Thus, determining the precise quantum grading is all that remains.

We know that for $D'$, this group appears in quantum grading $n' - |s_0 D'| + 3$. There is no difference in the homological gradings of the complexes of $D$ and $D'$, both of which are simply $|s|$. However, there is a difference in their quantum grading shifts. Since $|s_0 D|=|s_0 D'|$, their quantum grading shift differs by $n - n'$, that is, there is an additional shift of $n - n'$ when passing to the complex $CKh^{1,j}_{e,\mathrm{red}}(D)$.

Therefore, the group $\mathbb{Z}^{\,p_1(D)}$ must appear in quantum grading
\[
n - |s_0 D| + 3
\]
of $Kh^1_{e,\mathrm{red}}(D)$. Combining this with the fact that the Euler characteristic of a link can be computed from a positive diagram as $|s_0 D| - n$, this gives the desired result.
\end{proof}

%% file: Extreme_AlmostextremeKH.tex
\section{Extreme and almost extreme reduced even Khovanov homology} \label{sec: Extreme and almost extreme reduced even Khovanov homology}
In this section, we discuss the versions of the results proved in \cite{extremeKH} and \cite{almost_extremeKH} for extreme and almost extreme Khovanov homology, for reduced even Khovanov homology. 

\subsection{Extreme reduced even Khovanov homology} \label{sec: Extreme reduced even Khovanov homology}
In this section, we describe the reduced even Khovanov complex in the extreme quantum grading and use this description to determine its homotopy type by relating it to the descriptions of Khovanov homotopy types in \cite{LS_StableHomotopyType, extremeKH}.
\begin{proposition} \label{prop: extreme even KH complex}
Given a link diagram $D$ with a base point $p$, the even Khovanov chain complex and the reduced even chain complex are related as follows.
\[
\left(CKh_{e,\mathrm{red}}^{i,\, n_+ - 2n_- - |s_0 D| + 1}(D,p), \partial_e \right) = \left( CKh_{e}^{i,\, n_+ - 2n_- - |s_0 D|}(D), \partial_e \right).
\]
\end{proposition}

\begin{proof}
The proof is analogous to that of Lemma~\ref{lem: extreme even KH}.
The proof follows by explicitly understanding the generators of the Khovanov complex in this quantum grading. Consider the generator
	\[
	x = a_1 \otimes a_2 \otimes \cdots \otimes a_{|s_0 D|}
	\]
in $\Lambda_e(s_0 D)$, where each $a_i$ corresponds to a circle of $s_0 D$. Observe that this generator lies in the minimal possible quantum grading $n_+ - 2n_- - |s_0D| + 1$ and has homological grading $-n_-$.

To obtain a generator in the same quantum grading at homological grading $-n_-+1$, we must perform a $1$-smoothing at some crossing, which increases the number of circles. Thus, for such a state $s$, the corresponding generator is of the form
\[
x = a_1 \otimes a_2 \otimes \cdots \otimes a_{|sD|}
\]
in $\Lambda_e(sD)$, where each $a_i$ corresponds to a circle of $sD$.

More generally, each additional $1$-smoothing must increase the number of circles, with the generator given by the tensor product of the $x_i$ over all circles in that state.

Importantly, all of these generators are also generators of the reduced even chain complex. Therefore, we obtain an isomorphism
\[
CKh_{e,\mathrm{red}}^{i,\, n_+ - 2n_- - |s_0 D| + 1}(D,p) \;\cong\; CKh_{e}^{i,\, n_+ - 2n_- - |s_0 D|}(D).
\]
\end{proof}

The above proposition is simple, following directly from a basic understanding of the Khovanov chain complex, but it has several interesting consequences.

Before proceeding, we introduce the Spanier--Whitehead category, a simplified model of spectra that suffices for understanding the stable homotopy types arising from Khovanov homology.

\begin{definition}[Spanier--Whitehead category]
The Spanier--Whitehead category $\mathcal{SW}$ is defined as follows:
\begin{itemize}
    \item The objects are pairs $(X,n)$, where $X$ is a pointed CW-complex and $n \in \mathbb{Z}$.
    \item For two objects $(X_1,n_1)$ and $(X_2,n_2)$, the morphism set is
    \[
    \mathrm{Hom}_{\mathcal{SW}}\big((X_1,n_1),(X_2,n_2)\big)
    \coloneqq \varinjlim_{q \geq \max\{\lvert n_1\rvert,\lvert n_2\rvert\}}
    \big[\Sigma^{q+n_1} X_1, \Sigma^{q+n_2} X_2\big].
    \]
\end{itemize}
\end{definition}

The notion of reduced cohomology extends naturally to $\mathcal{SW}$ by defining
\[
\widetilde{H}^*\big((X,n)\big) \coloneqq \widetilde{H}^{*-n}(X).
\]

Let $D$ be a link diagram with a chosen basepoint $p$. In \cite{LS_StableHomotopyType}, R.~Lipshitz and S.~Sarkar associate to each quantum grading $j$ stable homotopy types $\mathcal{X}^j(D),\, \mathcal{X}^j_{\mathrm{red}}(D,p) \in \mathcal{SW}.$
These are defined so that their reduced cohomology recovers the even Khovanov homology.
\[
\widetilde{H}^*\big(\mathcal{X}^j(D)\big) \cong Kh^{*,j}(D), 
\qquad
\widetilde{H}^*\big(\mathcal{X}^j_{\mathrm{red}}(D,p)\big) \cong \widetilde{Kh}^{*,j}(D,p).
\]

In \cite{extremeKH}, J.~Gonz\'{a}lez-Meneses, P.~M.~G.~Manch\'{o} and M.~Silvero constructed another CW spectrum $X_D$ using Lando graphs and independence complexes, whose homology again recovered the even Khovanov homology of $D$ at the extreme quantum grading, i.e., $j= n_+ - 2n_- - |s_0 D|$. In \cite{ExtremeKhovanovSpectra}, it was shown that this construction is stable homotopy equivalent to the Lipshitz–Sarkar spectrum at the extreme quantum grading.

\begin{theorem}\label{thm:extreme_evenstablehomotopy}
Let $D$ be a link diagram with basepoint $p$. There is a stable homotopy equivalence at the minimal quantum grading:
\[
\mathcal{X}^{n_+ - 2n_- - |s_0 D| + 1}_{\mathrm{red}}(D,p) \;\simeq\; X_D \;\simeq\; \mathcal{X}^{n_+ - 2n_- - |s_0 D|}(D),
\]
where $X_D$ is the CW spectrum constructed in \cite{extremeKH}, and $\mathcal{X}^{j}_{\mathrm{red}}(D,p)$ and $\mathcal{X}^{j}(D)$ reduced and unreduced Lipshitz-Sarkar spectrum at quantum grading $j$, respectively.
\end{theorem}

\begin{proof}
These homotopy types are constructed from functors from the cube category to the Burnside category, which decompose as coproducts of functors indexed by the quantum grading. Proposition~\ref{prop: extreme even KH complex} shows that the functors corresponding to the extremal quantum grading are identical, and therefore determine identical homotopy types.
\end{proof}

\begin{remark}
This theorem shows that, at the minimal (extremal) quantum grading, the reduced and unreduced Khovanov stable homotopy types coincide with the CW spectrum $X_D$ constructed via Lando graphs. Hence, the extreme Khovanov homology captures the same stable homotopy type in all three constructions.
\end{remark}


\subsection{Almost extreme reduced even Khovanov homology}\label{sec: Almost extreme reduced even Khovanov homology}

This section is dedicated to proving Theorem~\ref{thm: maintheorem5} for even Khovanov homology.
The proof closely parallels the proof in Section~\ref{sec: Reduced even Khovanov homology of positive links}, with attention restricted to the almost-extreme quantum grading.


Before proceeding with the proof, we establish some preliminary notation and setup. 
Let $D$ be a $0$-adequate link diagram. 
Without loss of generality, we may assume that the base point $p$ lies away from all crossings and their resolutions. Let $D'$ be an unoriented link diagram with $n'$ crossings obtained from $D$ by performing a $0$-smoothing at all but one crossing between each pair of circles in $s_0 D$. We keep the base point $p$ fixed in its original position. Note that $D'$ is also $0$-adequate and satisfies $|s_0D'|=|s_0D|$.

We now define the complexes $C^{i,j}(D)$ and $C^{i,j}_{\red}(D,p)$ for unoriented link diagrams. The generators of $C^{i,j}(D)$ (respectively, $C^{i,j}_{\red}(D,p)$) are the same as those of the (unreduced) even Khovanov complex $CKh_e(D)$ (respectively, the reduced even Khovanov complex $CKh_{e,\red}(D,p)$), but the gradings are defined differently.

For an element $x = a_1 \otimes \cdots \otimes a_k \in \Lambda_e(sD)$, the gradings in the unreduced complex $C^{i,j}(D)$ are defined by
$$
h'(x) \coloneq |s| \quad \text{and} \quad q'(x) \coloneq |sD| - 2k + |s|.
$$

In the reduced complex $C^{i,j}_{\red}(D,p)$, the gradings are given by
$$
h'_{\red}(x) \coloneq |s| \quad \text{and} \quad q'_{\red}(x) \coloneq |sD| - 2k + |s| + 1.
$$

We denote the homology of the complexes $C^{i,j}(D)$ and $C^{i,j}_{\red}(D,p)$ by $H^{i,j}(D)$ and $H^{i,j}_{\red}(D,p)$, respectively.

\begin{remark}
The complexes $C^{i,j}(D)$ and $C^{i,j}_{\red}(D,p)$ are the unnormalized complexes defined in \cite[Equation~(44)]{Khovanov_homology} for unoriented diagrams. Working with these complexes is natural in the present setting, as our results concern $0$-adequate diagrams rather than the underlying links, and we work with unoriented diagrams throughout this section. Furthermore, at the level of diagrams, the skein exact sequence continues to hold for the unnormalized complexes.
\end{remark}

\begin{lemma}\label{lem: even almost extreme technical1}
Given a $0$-adequate diagram $D$ with a basepoint $p$ and $(D',p)$ constructed as above, we have 
\[
C_{\red}^{1,-|s_0D|+2}(D, p) \cong C_{\red}^{1,-|s_0D|+2}(D', p).
\]
\end{lemma}
The proof of this lemma is essentially the same as that of \cite[Theorem~4.1]{KMMS_Khovanov_pos_fib}, with only minor modifications. We include the details here for completeness.
\begin{proof}
    There is a natural inclusion of chain complexes
	\[
	C^{i,j+1}_{\mathrm{red}}(D,p) \hookrightarrow C^{i,j}(D).
	\]
	We therefore describe cycles in \( C^{i,j}(D) \); the same description applies to the reduced complex.
	Let \( G = G_0(D) \) be the state graph associated to the state \( s_0 \), which assigns the \(0\)-marker to every crossing. Let \( E = \{ e_1, \ldots, e_n \} \) be the edge set of \( G \), where each edge \( e_j \) corresponds to a crossing \( c_j \). Let \( l = |s_0 D| \) denote the number of circles in the resolution \( s_0 D \).
	
	The group \( C^{1}(D) \) is generated by states with exactly one \(1\)-marker, and decomposes as
	\[
	C^{1}(D) = \bigoplus_{i=1}^n A_i
	\]
	where \( A_i =\Lambda_e(s^iD)\) and the state \( s^i \) assigning a \(1\)-marker at \( c_i \) and \(0\)-markers elsewhere. 
    
    Similarly, the group $C^{2}(D)$ consists of states with exactly two $1$-markers
	\[
	C^2(D) = \bigoplus_{1 \leq i < j \leq n} A_{ij},
	\]
    where \( A_{ij} = \Lambda_e(s^{ij}D)\) and the state $s^{ij}$ assigns $1$-marker at crossings \( c_i \) and \( c_j \) and $0$-markers elsewhere. If the edges $e_i$ and $e_j$ are parallel in the graph $G$, meaning that the corresponding crossings connect the same pair of circles, then the resolution $s^{ij}D$ has exactly $l$ circles. Indeed, smoothing one of the crossings merges the two circles (since $D$ is $0$-adequate) and decreases the number of circles by one, while smoothing the second crossing splits this circle into a new two circles, restoring the total number of circles to $l$.

    The differential $d^1$ arises from the maps 
    $$\delta_{kij} \colon A_k \to A_{ij}.$$
    in the cube of resolution. We have by definition that \( \delta_{kij} = 0 \) whenever \( k \neq i, j \), as there is no edge connecting the corresponding states in the cube. The rest of the proof boils down to analysing this map when the edges \( e_i \) and \( e_j \) are parallel. 

    The only non-zero maps into  \( A_{ij} \) are $\delta_{iij} \colon A_i \to A_{ij}$, $\delta_{jij} \colon A_j \to A_{ij}$, see the following commutative diagram:
\begin{center}
	\begin{xy}
		(0,90)*+{}="empt";(30,90)*+{C^1(D)}="D1"; (90,90)*+{C^2(D)}="D2";%
		(30,70)*+{A_i\oplus A_j}="AiAj"; (90,70)*+{A_{ij}}="Aij";%
		{\ar@{->}^{d^1} "D1";"D2"};
		{\ar@{->}_{\pi_i\oplus \pi_j} "D1";"AiAj"};
		{\ar@{->}^{\pi_{ij}} "D2";"Aij"};
		{\ar@{->}^{\delta_{iij}\oplus\delta_{jij}} "AiAj";"Aij"};
	\end{xy}
\end{center}
Here, $\pi_i$, $\pi_j$, and $\pi_{ij}$ denote the natural projections onto the respective summands $A_i$, $A_j$, and $A_{ij}$. These smoothings share $l-2$ common circles, which we identify when comparing the corresponding states and their resolutions. With this identification, we obtain the following description of the maps.
\begin{align*}
\delta_{iij}(x) = \phantom{-} (a_p + 
a_q)\otimes x,\\
\delta_{jij}(x) = - (a_p + 
a_q) \otimes x,
\end{align*}
where $p$ and $q$ denote the circles of $s^{ij}D$ that are not shared by $s^iD$ and $s^jD$. The minus sign in $\delta_{jji}$ appears because, without loss of generality, we assume that $i < j$.

Now suppose that $x\in\ker(d^1)$, and set $x_i:=\pi_i(x)$ and $x_j:=\pi_j(x)$. Since $\pi_{ij}\circ d^1(x)=0$, it follows that
$$
\delta_{iij}(x_i)+\delta_{jij}(x_j)=0.
$$

Since $\delta_{iij}$ and $\delta_{jij}$ are injective, it follows that
\[
x_i = x_j.
\]

This argument shows that for any pair of parallel edges $e_i$ and  $e_j$ in $G$, the components $x_i$ and $x_j$ of any $x \in \ker(d^1)$ must agree under identification via the corresponding isomorphisms. Hence, the kernel of $d^1$ is contained in the subspace of $C^1(D)$ consisting of tuples $(x_1, \ldots, x_n)$ such that $x_r = x_s$ whenever $e_r$ and $e_s$ are parallel.

Thus, we may reduce to a diagram $D'$ by keeping exactly one crossing between each pair of circles in $s_0D$ and resolving all other crossings using the $0$-marker. The graph $G_{s_0}(D')$ is then $G^{\mathrm{red}}$. Now, this gives the natural isomorphism between $C^{1, -|s_0D|+1}(D)$ and $C^{1, -|s_0D|+1}(D')$.

Importantly, the above isomorphism restricts to an isomorphism of the reduced complexes, yielding the required result.

\end{proof}

\begin{proof}[Proof of \ref{thm: maintheorem5} (even case)]
We first consider the cases where $i = -n_-$ and $i \geq -n_- + 2$. In these ranges, we give an explicit description of the generators of the (unreduced) even Khovanov complex. We begin with the case $i = -n_-$, which splits into two subcases depending on whether the graph $G_0(D)$ is bipartite.

\begin{enumerate}
    \item \textbf{$G_0(D)$ is not bipartite}.  
    By \cite[Remark~6.4]{almost_extremeKH}, the (unreduced) even Khovanov complex has no cycles. Consequently, the reduced complex also has no cycles.

    \item \textbf{$G_0(D)$ is bipartite}.  This case proceeds identically to Lemma~\ref{lem: almost extreme even KH}, where the homological and quantum gradings are shifted by the presence of a negative crossing while the rest of the argument remains unchanged.
\end{enumerate}

We now consider the case $i \geq -n_- + 2$. In the quantum grading under consideration, each state $s$ with $|s| \geq 1$ contributes at most one generator to
$$
CKh_{e}^{-n_- + |s|,\, n - |s_0 D| + 3}(D).
$$

We describe precisely when such a generator appears.
\begin{enumerate}
    \item If $|s| = 1$, then $\Lambda_e(sD)$ contains a generator in this quantum grading, given by the tensor product of the element $a$ over all circles in the smoothing $sD$.
    \item If $|s| \geq 2$, a generator appears only if every lower state $s' < s$ with $|s'| \geq 2$ corresponds to a circle-splitting. In this case, the generator is uniquely determined and is again the tensor product of the element $a$ over all circles in the smoothing $sD$.
\end{enumerate}

Consequently, for $i \geq n_-+1$, the reduced chain complex $CKh_{e,\red}^{i,\, n_+ - 2n_- - |s_0 D| + 3}(D,p)$ coincides with the unreduced chain complex $CKh_{e}^{i,\, n_+ - 2n_- - |s_0 D| + 2}(D,p).$
The desired result now follows from \cite[Remark~6.4]{almost_extremeKH}.

The only case that remains is the one where $i=-n_-+1$. This follows by simply combining Lemma~\ref{lem: even almost extreme technical1} and Lemma~\ref{lem: even almost extreme technical2}.
At this point, we make essential use of the setup established at the beginning of this section.
\end{proof}

\begin{lemma}\label{lem: even almost extreme technical2}
The diagram $(D',p)$ constructed as above satisfies
\[
C_{\red}^{1,-|s_0D|+2}(D', p) = \Z^{p_1(D)}.
\]
\end{lemma}

\begin{proof}
	The proof of this lemma is essentially the same as that of Lemma~\ref{lem: reduced even KH last_group}, with only minor modifications. We proceed by induction on $p_1(D')$.
	\begin{itemize}
		\item If $p_1(D') = 0$, then $G_0(D')$ is a tree, so $D'$ represents the unknot. 
		By construction, the diagram $D'$ requires that any two Seifert circles share at most a single crossing, from which the claim follows immediately.
		\item Now suppose that $p_1(D') > 0$. Let $v$ be a crossing of $D'$, and consider its $0$-smoothing $D_0$ and $1$-smoothing $D_1$. Then, there is a long exact sequence of the complexes $C^{i,j}_{\red}$
		\[
		\longrightarrow 
		C_{\mathrm{red}}^{0, -|s_0D_1|+1}(D_1,p)
		\longrightarrow C_{\mathrm{red}}^{1, -|s_0D'| + 3}(D',p)
		\longrightarrow C_{\mathrm{red}}^{1, -|s_0D_0| + 3}(D_0,p) \longrightarrow.
		\]
		\item Since $D_0$ is also $0$-adequate, the induction hypothesis applies. Combining this with the results established already in the previous proof,
		\[
		C_{\mathrm{red}}^{1, -|s_0D_0| + 3}(D_0,p) \cong \mathbb{Z}^{\,p_1(D_0)} = \mathbb{Z}^{\,p_1(D') - 1}, \quad
		C_{e,\mathrm{red}}^{0, -|s_0D_0| + 3}(D_0,p) = 0.
		\] 
		Additionally, we have by Proposition~\ref{prop: extreme even KH complex}
		\[
		C_{\mathrm{red}}^{0,-|s_0D_1|+1}(D_1,p) \cong \mathbb{Z}, \quad C_{\mathrm{red}}^{1,-|s_0D_1|+1}(D_1,p) \cong 0,
		\]
	\end{itemize}
	Combining these facts with the long exact sequence, we obtain the short exact sequence
	\[
	0 \longrightarrow \mathbb{Z} \longrightarrow C_{\mathrm{red}}^{1, -|s_0D'| + 3}(D',p) \longrightarrow \mathbb{Z}^{\,p_1(D') - 1} \longrightarrow 0.
	\]
	This completes the induction step, proving the desired result.
\end{proof}

%% file: RedOddKH_PositiveLinks.tex
\section{Reduced odd Khovanov homology of positive links}\label{sec: Reduced odd Khovanov homology of positive links}
In this section, we establish Theorem~\ref{thm: maintheorem1} for odd Khovanov homology. The arguments largely parallel those of Section~\ref{sec: Reduced even Khovanov homology of positive links}.
The main objective of this section is to identify precisely where the sign choices in odd Khovanov homology play a role, and to demonstrate that the proofs adapt accordingly.

Let $L$ be a non-split positive link with a base point $p$. Choose a positive diagram $D$ of $L$ and, without loss of generality, assume that $p$ lies away from all crossings and their resolutions. Construct the diagram $D'$ as in Section~\ref{sec: Reduced odd Khovanov homology of positive links}. We also fix a decoration, i.e., an orientation at each crossing. The orientations are chosen so that, in the resolution $s_0D$, all arcs connecting a given pair of Seifert circles are consistently oriented from one circle to the other. 

\begin{proposition}\label{prop: odd technical}
The reduced odd Khovanov homology of the diagram $(D', p)$ constructed above satisfies
\[
Kh^{1}_{o, \mathrm{red}}(D,p) = Kh^{1}_{o, \mathrm{red}}(D',p).
\]

\end{proposition}

\begin{proof}
	The proof is analogous to that of Proposition~\ref{prop: even technical}, and we use the notation for states, state graphs, and the associated quantities introduced there. As in that proof, we suppress quantum grading shifts throughout.
	
	As recalled in Section~\ref{sec: reduced Khovanov homology}, there is a natural inclusion of chain complexes
	\[
	CKh^{i}_{o,\mathrm{red}}(D,p) \hookrightarrow CKh^{i}_o(D).
	\]
	We therefore describe cycles in \( CKh^{i}_o(D) \); the same description applies to the reduced complex.

	The group \( CKh^1_o(D) \) is generated by states with exactly one \(1\)-marker, and decomposes as
	\[
	CKh^1_o(D) = \bigoplus_{i=1}^n B_i,
	\]
	where \( B_i = \Lambda_o(s^iD)\) and the state \( s^i \) assigns a \(1\)-marker at \( c_i \) and \(0\)-markers elsewhere. 
	Similarly, the group \( CKh^2(D) \) consists of states with exactly two $1$-markers
	\[
	CKh^2_o(D) = \bigoplus_{1 \leq i < j \leq n} B_{ij},
	\]
	where \( B_{ij} = \Lambda_o(s^{ij}D)\) and the state $s^{ij}$ assigns $1$-marker at crossings \( c_i \) and \( c_j \) and $0$-markers elsewhere. 
	The differential \( d^1 \) arises from the maps 
\[
\delta_{kij} \colon B_k \to B_{ij},
\]
in the cube of resolutions. We have by definition that \( \delta_{kij} = 0 \) whenever \( k \neq i, j \), as there is no edge connecting the corresponding states in the cube. The rest of the proof boils down to analysing this map when the edges \( e_i \) and \( e_j \) are parallel. 

The only non-zero maps into  \( B_{ij} \) are $\delta_{iij} \colon B_i \to B_{ij}$, $\delta_{jij} \colon B_j \to B_{ij}$, see the following commutative diagram.
\begin{center}
	\begin{xy}
		(0,90)*+{}="empt";(30,90)*+{CKh_o^1(D)}="D1"; (90,90)*+{CKh_o^2(D)}="D2";%
		(30,70)*+{B_i\oplus B_j}="AiAj"; (90,70)*+{B_{ij}}="Aij";%
		{\ar@{->}^{d^1} "D1";"D2"};
		{\ar@{->}_{\pi_i\oplus \pi_j} "D1";"AiAj"};
		{\ar@{->}^{\pi_{ij}} "D2";"Aij"};
		{\ar@{->}^{\delta_{iij}\oplus\delta_{jij}} "AiAj";"Aij"};
	\end{xy}
\end{center}
Here, $\pi_i$, $\pi_j$, and $\pi_{ij}$ denote the natural projections onto the respective summands $B_i$, $B_j$, and $B_{ij}$. These smoothings share $l-2$ common circles, which we identify when comparing the corresponding states and their resolutions. With this identification, we obtain the following description of the maps.
\begin{align*}
\delta_{iij}(x) &= \phantom{-} (v_p-v_q)\otimes x,\\
\delta_{jij}(x) &= - (v_p-v_q) \otimes x,
\end{align*}
where $p$ and $q$ are the two circles in $s^{ij}$ that are not common, and the minus sign in $\delta_{jij}$ arises because, with our choice of decoration, these crossings form a type-A square (see Figure~\ref{fig:SquareTypes}).

To construct the chain complex, we first fix an edge assignment and a sign assignment so that each square anti-commutes. In a type-A square, a $(-1)$ sign appears on even number of edges. There are two cases, depending on whether both of $\delta_{iij}$ and $\delta_{jij}$ carry a sign or only one does. Since the proofs of these cases are essentially identical, we treat them simultaneously.

Let $x \in \ker(d^1)$, and write $x_i := \pi_i(x)$ and $x_j := \pi_j(x)$. For $k \neq i,j$, we have $\delta_{kij} = 0$, so the contribution of $d^1(x)$ to $A_{ij}$ is
\[
\delta_{iij}(x_i) \pm \delta_{jij}(x_j).
\]
The condition $\pi_{ij} \circ d^1(x) = 0$ then implies
\[
\delta_{iij}(x_i) \pm \delta_{jij}(x_j) = 0.
\]

Since $\delta_{iij}$ and $\delta_{jij}$ are injective, it follows that
\[
x_i = \pm x_j.
\]

Consequently, for any pair of parallel edges $e_i$ and $e_j$, the corresponding components $x_i$ and $x_j$ of $x \in \ker(d^1)$ must agree up to a sign. This sign is determined by the edge and sign assignments chosen for $D$. Therefore, $\ker(d^1)$ is contained in the subspace of $CKh^1_o(D,p)$ consisting of tuples $(x_1, \dots, x_n)$ satisfying
\[
x_r = \pm x_s \quad \text{whenever $e_r$ and $e_s$ are parallel,}
\]
with the signs determined by the chosen edge and sign assignments.

It follows that, for the purpose of computing $Kh^1_o(D)$, we may replace $D$ by a diagram $D'$ obtained by retaining exactly one crossing between each pair of circles in $G_0(D)$ and resolving all remaining crossings using the $0$-marker. In particular,
\[
Kh^1_o(D,p) \cong Kh^1_o(D',p).
\]
\end{proof}

\begin{lemma} \label{lem: extreme odd KH}
If $D$ is $0$-adequate link diagram with $n_+$ positive and $n_-$ negative crossings, then its reduced odd Khovanov homology satisfies the following.
\[
Kh_{o,\red}^{i,n_+ - 2n_- - |s_0 D| + 1}(D,p) =
\begin{cases}
	\mathbb{Z} & \text{if } i = -n_- \\
	0 & \text{otherwise.}
\end{cases}
\]
\end{lemma}

\begin{proof}
	The argument of this proof is identical to that of Lemma~\ref{lem: extreme even KH}.
\end{proof}

\begin{lemma} \label{lem: almost extreme odd KH}
If $D$ is a positive diagram of $L$ with $n$ crossings, then its reduced odd Khovanov homology satisfies the following.
\[
Kh_{o,\red}^{0,n - |s_0 D| + 3}(D,p) =0.
\]
\end{lemma}

\begin{proof}
	The argument of this proof is identical to that of Lemma~\ref{lem: almost extreme even KH}.
\end{proof}

\begin{lemma}\label{lem: reduced odd KH last_group}
 The reduced odd Khovanov homology of the diagram $(D',p)$ above satisfies 
	\begin{align*}
		Kh^{1, j}_{o,\red}(D',p)=\begin{cases}
			\Z^{p_1(D')} &\textrm{ if } j=n'-|s_0D|+3,\\
			0 &\textrm{ otherwise,}
		\end{cases}
	\end{align*}
\end{lemma}
\begin{proof}
	The proof is identical to that of Lemma~\ref{lem: reduced even KH last_group}, the only additional input being that Lee’s theorem extends to odd Khovanov homology.
\end{proof}

\begin{proof}[Proof of Theorem~\ref{thm: maintheorem1} (odd case)]
From Proposition~\ref{prop: odd technical} and Lemma~\ref{lem: reduced odd KH last_group}, we obtain that the reduced even Khovanov homology of $D$ in homological grading $1$ is $\mathbb{Z}^{\,p_1(D)}$. Thus, determining the precise quantum grading is all that remains.
The proof of the first part of the theorem is identical to the argument given for the even case in Section~\ref{sec: Reduced even Khovanov homology of positive links}.
\end{proof}

%% file: Extreme_AlmostextremeOddKH.tex
\section{Extreme and almost extreme reduced odd Khovanov homology} \label{sec: Extreme and almost extreme reduced odd Khovanov homology}
In this section, we discuss analogues of the results from \cite{almost_extremeKH} on almost extreme Khovanov homology in the setting of reduced odd Khovanov homology. We also establish a partial generalization of the results from Section~\ref{sec: Extreme reduced even Khovanov homology} to the odd theory.

As before, we pinpoint where the sign choices in the odd theory become relevant and verify that the arguments carry over.

\subsection{Extreme reduced odd Khovanov homology} \label{sec: Extreme reduced odd Khovanov homology}
\begin{proposition} \label{prop: extreme odd KH complex}
Given a link diagram $D$ with a base point $p$, the odd Khovanov chain complex and the reduced odd chain complex are related as follows.
\[
\left(CKh_{o,\mathrm{red}}^{i,\, n_+ - 2n_- - |s_0 D| + 1}(D,p), \partial_o \right) = \left( CKh_{o}^{i,\, n_+ - 2n_- - |s_0 D|}(D), \partial_o \right).
\]
\end{proposition}

\begin{proof}
The proof of this proposition is identical to that of Proposition~\ref{prop: extreme even KH complex}.
\end{proof}

Given a link diagram $D$ with a basepoint $p$, S.~Sarkar, C.~Scaduto and M.~Stoffregen~\cite{SSS_OddStableHomotopyType} constructed stable homotopy types 
\(\mathcal{X}_{o}^j(D)\) and \(\mathcal{X}^j_{o,\mathrm{red}}(D,p)\) for each quantum grading $j$, whose homology recovers the (unreduced and reduced) odd Khovanov homology, respectively.
\[
\widetilde{H}_*(\mathcal{X}_o^j(D)) \cong Kh_o^{*,j}(D), 
\qquad
\widetilde{H}_*(\mathcal{X}^j_{o, \mathrm{red}}(D,p)) \cong Kh^{*,j}_{o,\red}(D).
\]

The isomorphism of chain complexes established in Proposition~\ref{prop: extreme odd KH complex} immediately yields the following result.

\begin{theorem}
Let $D$ be a link diagram with basepoint $p$. There is a stable homotopy equivalence at the minimal quantum grading:
\[
\cX^{n_+ - 2n_- - |s_0 D| + 1}_{o,\red}(D,p) \;\simeq\; \cX_o^{n_+ - 2n_- - |s_0 D|}(D),
\]
where $\cX^{j}_{o, \red}(D,p)$ and $\cX_o^{j}(D)$ reduced and unreduced odd Khovanov spectrum at quantum grading $j$, respectively.
\end{theorem}

\begin{proof}
The proof is identical to that of Theorem~\ref{thm:extreme_evenstablehomotopy}.
\end{proof}


\subsection{Almost extreme reduced odd Khovanov homology}\label{sec: Almost extreme reduced odd Khovanov homology}
This section is dedicated to proving Theorem~\ref{thm: maintheorem5} for odd Khovanov homology.

\begin{proof}[Proof of Theorem~\ref{thm: maintheorem5} (odd case)]
We first consider the cases where $i = -n_-$ and $i \geq -n_- + 2$. We begin with the case $i = -n_-$, which splits into two subcases depending on whether the graph $G_0(D)$ is bipartite.

\begin{enumerate}
	\item \textbf{$G_0(D)$ is bipartite}.  
	There is a unique generating cycle in the unreduced odd Khovanov complex. It is given by
	\[
	\sum_{v \in V_{G_0(D)}} \sigma(v)\, x_1 \otimes \cdots \widehat{x_v} \cdots \otimes x_{|s_0|-1},
	\]
	where $\sigma \colon V_{G_0(D)} \to \{\pm 1\}$ assigns opposite signs to any pair of adjacent vertices, and $\widehat{x_v}$ denotes omission of the corresponding tensor factor.

	This form is dictated by the quantum grading, which requires each generator to omit exactly one factor $x_v$, where $v\in V_{G_0(D)}$. Since all summands of the differential arise from merge maps, if two vertices $v$ and $w$ are joined by an edge corresponding to a crossing $c$, then omitting either $x_v$ or $x_w$ yields the same generator in $\fF_o(s)$, where $s$ is the state such that $s(c)=1$ and $s(c')=0$ for all other crossings $c'$. Thus, the corresponding terms must cancel, which is achieved by assigning opposite signs to adjacent vertices. This determines $\sigma$ uniquely up to an overall sign.
	

    However, this element does not lie in the subcomplex
    $$CKh_{o,\red}^{-n_-,\, n_+ - 2n_- - |s_0 D| + 3}(D,p) \hookrightarrow CKh_{o}^{-n_-,\, n_+ - 2n_- - |s_0 D| + 2}(D,p),$$
    and therefore does not contribute to the reduced odd Khovanov homology.

    \item \textbf{$G_0(D)$ is not bipartite}.  
    The unreduced odd Khovanov complex admits no nontrivial cycles. Indeed, the description of potential cycles would require the existence of a function 
	\[
	\sigma \colon V_{G_0(D)} \to \{\pm 1\}
	\]
	that assigns opposite signs to any pair of adjacent vertices. However, no such assignment exists if $G_0(D)$ is not bipartite, and hence it is impossible to arrange cancellations among the terms in the differential. It follows that no nontrivial cycle can occur. Consequently, the reduced complex also admits no cycles.
\end{enumerate}

We now turn to the case $i\geq -n_-+2$. We first construct an isomorphism of complexes
\begin{equation}\label{eqn:odd_even_isomorphism}
CKh_{e}^{-n_- + |s|,\, n_+-2n_- - |s_0 D| + 3}(D) \cong CKh_{o}^{-n_- + |s|,\, n_+ - 2n_- - |s_0 D| + 3}(D)\quad \text{for}\ |s|\geq 1.
\end{equation}
Assuming this isomorphism, we can complete the proof for $i\geq -n_-+2$. Indeed, the description of the generators for $s\geq1$ given in the proof of Theorem~\ref{thm: maintheorem5} in Section~\ref{sec: Almost extreme reduced even Khovanov homology} carries over to the odd Khovanov setting. In particular, for $i\geq1$, the reduced chain complex $CKh_{o,\red}^{i,n_+-2n_--|s_0D|+3}(D,p)$ coincides with the corresponding unreduced chain complex $CKh_{o}^{i,n_+-2n_--|s_0D|+2}(D,p)$.
The result now follows from Theorem~\ref{thm: maintheorem5} in the even case and Equation~\eqref{eqn:odd_even_isomorphism}.

We construct an isomorphism of complexes
\[
CKh_{e}^{-n_- + |s|,\, n_+ - 2n_- - |s_0 D| + 3}(D) \cong CKh_{o}^{-n_- + |s|,\, n_+ -2n_- - |s_0 D| + 3}(D)
\]
for $|s| \geq 1$. 
Fix orientations on the crossings so that, in the resolution $s_0D$, all arcs connecting any given pair of circles in $G_0(D)$ are directed consistently from one circle to the other. We now describe a compatible ordering of the circles in $sD$ for $|s|\geq 1$. This ordering allows us to write down explicit generators in the quantum grading $n_+ -2n_- - |s_0 D| + 3$. 


Given a state $s$ with $|s|=1$, there is a unique generator in $\fF_o(sD)$ in the required quantum grading, given by the tensor product of the elements $a$ over all circles in the smoothing $sD$. The boundary of this generator has nonzero contributions only in those summands $\fF_o(s'D)$ where $s'$ is obtained by resolving an additional crossing parallel to the one already resolved.

Consequently, when restricting to the complex
\[
CKh_{o}^{-n_- + |s|,\, n_+ -2n_- - |s_0 D| + 3}(D)
\quad \text{for } |s| \geq 1,
\]
the complex decomposes into summands corresponding to crossings between fixed pairs of circles. It therefore suffices to carry out the construction for crossings between a fixed pair of circles.

Fix such a pair. We first choose and fix an ordering of all crossings not occurring between these two circles. Among the crossings between the chosen pair, select one as a reference and order the remaining crossings as illustrated in Figure~\ref{fig:ordering1}.

\begin{figure}[h!]
	\centering
	\includegraphics[width=.25\textwidth]{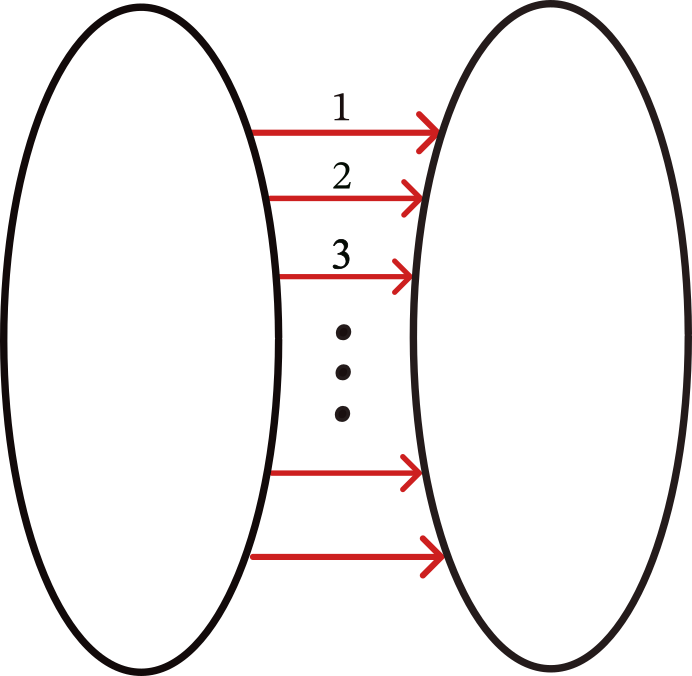}
	\caption{Ordering of crossings.}
	\label{fig:ordering1}
\end{figure}

With these conventions, we define the ordering of circles in $sD$ as follows. For $|s|=1$, the circle obtained by merging the two chosen circles is placed first, and the remaining circles follow in the fixed order. For $|s|=2$, we label the newly created circle by $x_2$ and the circle inherited from the previous merging by $x_1$, so that the ordering is $x_1, x_2$, followed by the remaining circles in the fixed order (See Figure~\ref{fig:ordering2}).

\begin{figure}[h!]
	\centering
	\includegraphics[width=.25\textwidth]{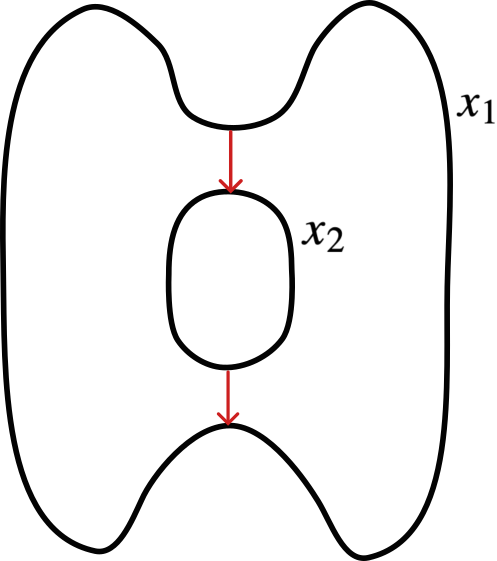}
	\caption{Ordering of circles.}
	\label{fig:ordering2}
\end{figure}

These choices determine the required identifications, and with respect to these orderings, the induced isomorphism between the even and odd complexes commutes with the boundary maps for $|s|=1,2$.

We proceed in the same manner for the remaining states with $|s|>2$. The reference crossings chosen for each pair of Seifert circles determine an ordering of the circles in each state. Specifically, we place the outermost circle first and order the remaining circles from top to bottom, with the reference crossings determining this order. The resulting isomorphisms between the even and odd complexes commute with the boundary maps for every state with $|s|>1$.

The Seifert circles need not be arranged as in Figure~\ref{fig:ordering2}; one may lie inside another. The same argument applies in this case. We orient all crossings so that the corresponding arcs are consistently directed from one Seifert circle to the other. We then choose the ordering of the crossings so that, in the $1$-resolution, the induced directions on the resulting arcs are compatible with this ordering. With this choice, the argument above carries over unchanged.

The only case that remains is the one where $i=-n_-+1$. This follows by simply combining Lemmas~\ref{lem: odd almost extreme technical1} and \ref{lem: odd almost extreme technical2}.
\end{proof}

Construct the diagram $D'$ as in Section~\ref{sec: Almost extreme reduced even Khovanov homology}. We then define the complexes $C_o^{i,j}(D)$ and $C^{i,j}_{o,\red}(D,p)$ as in the same section, replacing the even Khovanov complexes with their odd counterparts.

\begin{lemma}\label{lem: odd almost extreme technical1}
Given a $0$-adequate diagram $D$ with a basepoint $p$ and $(D',p)$ constructed as above, we have 
\[
C_{o,\red}^{1,-|s_0D|+2}(D, p) \cong C_{o,\red}^{1,-|s_0D|+2}(D', p).
\]
\end{lemma}

\begin{proof}
 The proof of this lemma is identical to that of Proposition~\ref{prop: odd technical}, with minor modifications as in Lemma~\ref{lem: even almost extreme technical1}.
\end{proof}

\begin{lemma}\label{lem: odd almost extreme technical2}
The diagram $(D',p)$ constructed as above has the following property
\[
C_{o,\red}^{1,-|s_0D|+2}(D', p) = \Z^{p_1(D)}
\]
\end{lemma}
\begin{proof}
The proof of this lemma is identical to that of lemma~\ref{lem: even almost extreme technical2}.
\end{proof}

%% file: General_coefficients.tex
\section{Reduced Khovanov homology over other coefficent rings} \label{sec: Reduced Khovanov homology over other coefficent rings}

Torsion phenomena in the Khovanov homology of positive links were studied in \cite{KMMS_Khovanov_pos_fib}, motivating the study of torsion in the reduced even and odd theories. In the unreduced setting, a key ingredient is \cite[Proposition~4.8]{semi_adequate_torsion}, for which there is no direct analogue in the reduced setting. Nevertheless, the reduced theory remains sufficiently tractable to carry out the analysis over arbitrary coefficient rings. 
We first establish the following proposition for general coefficient rings and then apply the arguments of \cite[Section~4.2]{KMMS_Khovanov_pos_fib} to obtain the corresponding torsion results.

\begin{proposition} \label{prop: general coefficent}
Let $L$ be a non-split positive link with base point $p$, and let $R$ be a commutative ring with unit. Then the reduced even and odd Khovanov homologies over $R$ satisfy
\[
\begin{aligned}
Kh^{1,j}_{*,\red}(L,p;R) &\cong
\begin{cases}
R^{p_1(L)} & \text{if } j = 3 - \chi(L),\\
0 & \text{otherwise,}
\end{cases} \\
Kh^{i,\,3-\chi(L)}_{*,\red}(L,p;R) &\cong
\begin{cases}
R^{p_1(L)} & \text{if } i = 1,\\
0 & \text{otherwise.}
\end{cases}
\end{aligned}
\]
where $* \in \{e,o\}$, and $\chi(L)$ and $p_1(L)$ are as previously defined.
\end{proposition}

\begin{proof}
	The proof proceeds by verifying that the arguments of the previous sections extend to this general setting.	

	\begin{enumerate}
	\item In Section~\ref{sec: Reduced even Khovanov homology of positive links}, Proposition~\ref{prop: even technical} and Lemma~\ref{lem: extreme even KH} are independent of the choice of coefficient ring; their statements and proofs remain valid without modification. For Lemma~\ref{lem: almost extreme even KH}, the argument carries over directly and shows that the relevant cycle does not lie in the reduced complex. Furthermore, Lemma~\ref{lem: reduced even KH last_group} relies only on exact sequences and Lee's theorem for alternating links, both of which hold over arbitrary coefficient rings. 
	Thus, the arguments in this section extend without essential modification to give the first statement of the theorem in the even case. 
	
	Similarly, the arguments of Section~\ref{sec: Reduced odd Khovanov homology of positive links} extend to arbitrary coefficient rings by the same reasoning as above, yielding the first statement in the odd case.

	\item The second statement follows from an application of the universal coefficient theorem to Proposition~\ref{thm: mainproposition7}.
\end{enumerate}
\end{proof}

An immediate consequence is the following.

\begin{corollary}\label{cor: torsion}
Let $L$ be a non-split positive link with base point $p$. Then
\[
\tor\bigl(Kh^{2,j}_{*,\red}(L,p)\bigr) = 0 \quad \text{for all } j,
\]
where $*\in \{e,o\}$ and $\tor$ denotes the torsion subgroup.
\end{corollary}

\begin{proof}
The universal coefficient theorem gives a short exact sequence
\[
0 \to Kh^{1,j}_{*,\red}(L,p)\otimes R
\to Kh^{1,j}_{*,\red}(L,p;R)
\to \tor^{\mathbb{Z}}_{1}\bigl(Kh^{2,j}_{*,\red}(L,p),R\bigr)
\to 0.
\]

\begin{enumerate}
	\item For $j = 1-\chi(L)$, the result follows from Lemmas~\ref{lem: extreme even KH} and \ref{lem: extreme odd KH}.
	\item For $j = 3-\chi(L)$, the result follows from Proposition~\ref{thm: mainproposition7}. 
	\item For $j > 3-\chi(L)$, both $Kh^{1,j}_{*,\red}(L,p;R)$ and $Kh^{1,j}_{*,\red}(L,p)\otimes R$ vanish by Proposition~\ref{prop: general coefficent} and Theorem~\ref{thm: maintheorem1}. Hence
	\[ \tor^{\mathbb{Z}}_{1}\bigl(Kh^{2,j}_{*,\red}(L,p),R\bigr)=0 \]
	for all $R$, and therefore $\tor\bigl(Kh^{2,j}_{*,\red}(L,p)\bigr)=0.$
\end{enumerate}
\end{proof}

\begin{corollary} \label{cor: torsion_positive}
Let $L$ be a non-split positive link. Then
\[
\tor\bigl(Kh^{2,j}_{o}(L)\bigr) = 0 \quad \text{for all } j,
\]
where $\tor$ denotes the torsion subgroup.
\end{corollary}

\begin{proof}
We begin by choosing a base point $p$ on a link diagram. Using Corollary~\ref{cor: torsion} together with the relation
\[
Kh^{i,j}_o(L) \cong Kh^{i,j-1}_{o,\red}(L) \oplus Kh^{i,j+1}_{o,\red}(L),
\]
we obtain the desired result.
\end{proof}

\begin{proof}[Proof of Theorem~\ref{thm: maintheorem4}]
This follows immediately from Corollaries~\ref{cor: torsion} and~\ref{cor: torsion_positive}.
\end{proof}

%% file: Cables_OddKH.tex
\section{Reduced Khovanov homology of cables} \label{sec: Reduced Khovanov homology of cables}
In this section, we investigate cables of positive knots, drawing motivation from the results of \cite{KMMS_Khovanov_pos_fib}. Our approach closely parallels the arguments in \cite[Section~5]{KMMS_Khovanov_pos_fib}.

Given a diagram $D$ with writhe $w$, the standard diagram $D_{p,q}$ of the $(p,q)$-cable $K_{p,q}$ is obtained by taking $p$ blackboard parallel copies of $D$ and adding $n=pw-q$ negative $(1/p)$-twists to the $p$-parallel strands. Let $D$ be a positive diagram. If $q \geq pw$, then all crossings in the diagram $D_{p,q}$ are positive. Hence, the result follows by combining Theorem~\ref{thm: maintheorem1} with the observation that $p_1(K_{p,q}) = p_1(K)$ for $q \geq pw$, together with Schubert’s cabling formula~\cite{Schubert}
\[
-\chi(K_{p,q}) = -p\chi(K) + q(p - 1).
\] 

It therefore remains to establish the result for $p \leq q < pw$.
We proceed by induction. The inductive step requires the introduction of a family of satellite links, which we refer to as twisted cable links.

Let $2 \leq p \leq q \leq pw$ and $0 \leq m \leq p-1$. The \emph{twisted cable} $K_{p,q;m}$ is defined as the satellite link of $K$ with pattern given by the braid word
\[
w_{p,q} \cdot (\sigma^{-1} \sigma^{-2} \cdots \sigma^{-m}),
\]
where $w_{p,q} = (\sigma^{1}\sigma^{2}\cdots \sigma^{p-1})^q$ is the standard $p$-stranded braid word of the torus link $T(p,q)$.

\begin{lemma}\cite[Lemma~5.4]{KMMS_Khovanov_pos_fib}
    Let $K$ be a nontrivial knot and $p \neq 0$. Then the twisted cable $K_{p,q;m}$ of $K$ is a non-split link.
\end{lemma}

\begin{proposition}\label{prop:twisted_cable}
	Let $K$ be a positive knot. Then, for $2\leq p \leq q-1$ and $1\leq m \leq p-1$ we have 
	\begin{align*}
		Kh^{i,j}_{*,\red}(K_{p,q;m})&=0 \,\textrm{ for every }\,  i<0,\\
		Kh^{0,j}_{*,\red}(K_{p,q;m})&=\begin{cases}
			\Z\,&\textrm{ if }\,  j=1-p\chi(K)+q(p-1)-m, \\
			0 \,&\textrm{else,}
		\end{cases}\\
		Kh^{1,j}_{*,\red}(K_{p,q;m})&=\begin{cases}
			\Z^{p_1(K)}\,&\textrm{ if }\,  j=3-p\chi(K)+q(p-1)-m,\\
            0 \,&\textrm{else}
		\end{cases}\\
		\tor(Kh^{2,j}_{*,\red}(K_{p,q;m})) &=0 \quad \text{for every} \quad j.
	\end{align*}
\end{proposition}

\begin{proof}
The proof of this theorem proceeds as in \cite[Theorem~5.5]{KMMS_Khovanov_pos_fib}, since the underlying argument relies entirely on exact sequences that are equally valid in the reduced setting.
\end{proof}

\begin{proof}[Proof of Theorem~\ref{thm: maintheorem4}]
The proof naturally splits into two cases.
For $q \geq pw$, the cable $K_{p,q}$ is positive, the result follows by combining Theorem~\ref{thm: maintheorem1} with the observation that $p_1(K_{p,q}) = p_1(K)$ for $q \geq pw$, together with Schubert’s cabling formula~\cite{Schubert}
\[
-\chi(K_{p,q}) = -p\chi(K) + q(p - 1).
\] 
For $p+1 \leq q < pw$, the result follows from Proposition~\ref{prop:twisted_cable}. For the boundary case $p=q$, we use the identification $K_{p,q}=K_{p,q+1;p-1}$. Since $(q+1)-1 \geq p$, Proposition~\ref{prop:twisted_cable} applies to $K_{p,q+1;p-1}$. Furthermore,
$$1-p\chi(K)+(q+1)(p-1)-(p-1)=1-\chi(K_{p,q}).$$
Hence, the quantum gradings given by Proposition~\ref{prop:twisted_cable} agree with those in the statement of Theorem~\ref{thm: maintheorem5}.
\end{proof}